\pdfoutput=1 
\documentclass[12pt]{article}

\usepackage[margin=1in]{geometry}

\usepackage{amsmath,amsthm,amsfonts,mathrsfs}
\usepackage{authblk,enumerate}
\usepackage{subfig}
\usepackage{float}
\usepackage{url}
\usepackage{ifpdf}
\usepackage{graphicx}
\usepackage{epstopdf}
\usepackage{citesort}
\usepackage[toc, page]{appendix}

\numberwithin{equation}{section}
 
\newtheorem{theorem}{Theorem}[section]
\newtheorem{lemma}[theorem]{Lemma}

\newtheorem{definition}[theorem]{Definition}
\newtheorem{remark}[theorem]{Remark}

\newtheorem{corollary}[theorem]{Corollary}

\newcommand{\R}{\mathbb{R}}

\title{\bf Regularized integral  formulation of mixed Dirichlet-Neumann
  problems} \author[1]{\normalsize Eldar
  Akhmetgaliyev\thanks{eakhmetg@caltech.edu}} \author[2]{\normalsize
  Oscar P. Bruno\thanks{obruno@caltech.edu}} \affil[1,2]{\normalsize
  Computing \& Mathematical Sciences, California Institute of
  Technology} \date{} \date{\today}

\begin{document}

\maketitle

\begin{abstract}

  This paper presents a theoretical discussion as well as novel
  solution algorithms for problems of scattering on smooth
  two-dimensional domains under Zaremba boundary conditions---for
  which Dirichlet and Neumann conditions are specified on various
  portions of the domain boundary.  The theoretical basis of the
  proposed numerical methods, which is provided for the first time in
  the present contribution, concerns detailed information about the
  singularity structure of solutions of the Helmholtz operator under
  boundary conditions of Zaremba type.  The new numerical method is
  based on use of Green functions and integral equations, and it
  relies on the Fourier Continuation method for regularization of all
  smooth-domain Zaremba singularities as well as newly derived
  quadrature rules which give rise to high-order convergence even
  around Zaremba singular points. As demonstrated in this paper, the
  resulting algorithms enjoy high-order convergence, and they can be
  used to efficiently solve challenging Helmholtz boundary value
  problems and Laplace eigenvalue problems with high-order accuracy.

\end{abstract}

\newpage


\section{Introduction}\label{intro}

This paper concerns problems of scattering on smooth two-dimensional
domains under Zaremba boundary conditions, and, in particular, it
presents a theoretical discussion as well as novel solution algorithms
for this problem.  The theoretical basis of the proposed numerical
methods concerns detailed information, put forth in
Section~\ref{sing_struct}, about the singularity structure of the
solutions of the Helmholtz operator under boundary conditions of
Zaremba type.  The new numerical method, which is based on use of
Green functions and integral equations, incorporates use of the
Fourier Continuation method for regularization of all smooth-domain
Zaremba singularities as well as newly derived quadrature rules which
give rise to high-order convergence even around Zaremba singular
points. As demonstrated in this paper, the resulting algorithms enjoy
high-order convergence, and they can used to efficiently solve
challenging Helmholtz boundary value problems and Laplace eigenvalue
problems with high-order accuracy.

We thus consider the classical Zaremba boundary value problem
\begin{equation}
\begin{split}
\label{eq:exterior}
\Delta  u(x)+k^2 u(x)&=0 \qquad x\in \Omega, \\
u(x)&=f(x) \quad x\in\Gamma_D, \\
\displaystyle \frac{\partial  u(x)}{\partial  n_x}&=g(x) \quad x\in\Gamma_N\\
\end{split}
\end{equation}
for the Helmholtz equation, where $\Omega$ is a 2-dimensional domain
with boundary $\Gamma$ consisting of two disjoint portions $\Gamma_D$
and $\Gamma_N$ upon which Dirichlet and Neumann values are prescribed,
respectively; see Figure~\ref{fig:domain}.  

\begin{figure}[h!]
\centering
\includegraphics[width=2in]{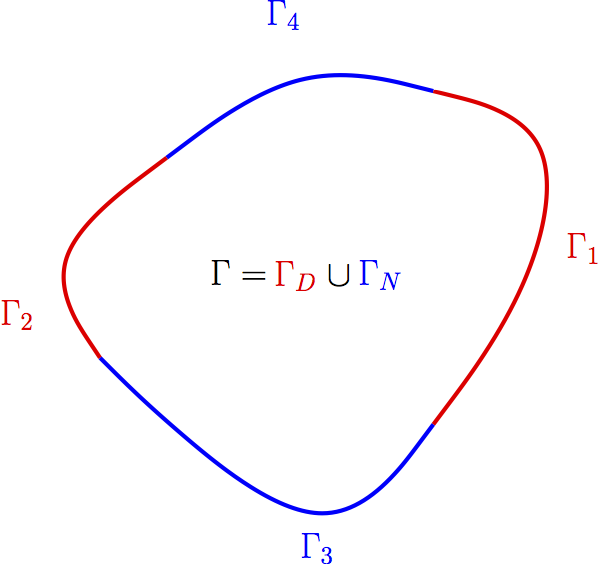}
\caption{Generic smooth domain $\Omega$ with boundary partitioned into
  Dirichlet and Neumann regions shown in red and blue, respectively.}
\label{fig:domain}
\end{figure}

Significant challenges arise in connection with Zaremba boundary value
problems in view of the singular character of its solutions: as first
established by Fichera~\cite{fichera1952sul}, Zaremba solutions are
generally non-smooth even for infinitely differentiable boundary data
$f$ and $g$, and smoothness of solutions can only be ensured provided
$f$ and $g$ obey certain relations which generically are not
satisfied.  Wigley~\cite{wigley1964,wigley1970mixed} provides a
detailed description of the singularity structure around Zaremba
points---which includes singularities of both square-root and
logarithmic type. In Section~\ref{sing_struct} of the present paper it
is shown, however, that a tighter result holds in the case the domain
boundary is itself smooth.

Problems with mixed Dirichlet-Neumann boundary conditions were first
considered in Zaremba's 1910 contribution~\cite{zaremba1910probleme},
which established existence and uniqueness of solution in the
particular case of the Laplace equation ($k=0$). Boundary conditions
of Zaremba type arise in a number of important areas, including
elasticity theory (were it appears as a model in the contexts of
contact mechanics~\cite{wendland1979integral} and crack
theory~\cite{fabrikant1991mixed}); homogenization theory (as it
applies to problems of steady state diffusion through perforated
membranes~\cite{fabrikant1991mixed}), etc. One of the main motivations
for our consideration of this problem concerns computational
electromagnetism: the Zaremba problem serves as a valuable stepping
stone to the closely related but more complex problem of
electromagnetic propagation and scattering at and around structures
such as printed circuit-boards. (As in the Zaremba problem, where the
boundary conditions change type at the Dirichlet-Neumann boundary, the
boundary conditions for the Maxwell equations at and around circuit-boards
change type---not from Dirichlet to Neumann but from dielectric
transmission conditions to perfect-conductor conditions---at the edges
of the perfectly conducting circuit elements.)

After the initial contribution by Zaremba, early works concerning the
Zaremba boundary value problem include results by
Signorini~\cite{signorini1916sopra} (1916: solution of the Zaremba
problem in the upper half plane using complex variable methods);
Giraud~\cite{giraud1934problemes} (1934: existence of solution of
Zaremba problems for general elliptic operators);
Fichera~\cite{fichera1949analisi,fichera1952sul} (1949, 1952:
regularity studies at Zaremba points, Zaremba-type problem for the
elasticity equations in two spatial dimensions);
Magenes~\cite{magenes1955sui} (1955: proof of existence and
uniqueness, single layer potential representation);
Lorenzi~\cite{lorenzi1975mixed} (1975: Sobolev regularity around a
corner which is also a Dirichlet-Neumann junction); and
Wigley~\cite{wigley1964,wigley1970mixed} (1964, 1970: explicit
asymptotic expansions around Dirichlet-Neumann junctions), amongst
others. More recent contributions in this area include
reference~\cite{wendland1979integral}, which provides a valuable
review in addition to a study of Zaremba singularities and theoretical
results concerning Galerkin-based computational approaches;
reference~\cite{cakoni2005boundary}, which considers the Zaremba
problem for the biharmonic equation;
references~\cite{duduchava2013mixed,duduchava2015mixed}, which study
Zaremba boundary value problems for Helmholtz and Laplace-Beltrami
equations; reference~\cite{chang2014edge}, which discusses the
solvability of the Zaremba problem from the point of view of
pseudo-differential calculus and Sobolev regularity theory;
reference~\cite{helsing2009} which introduces a certain inverse
preconditioning technique to reduce the number of linear algebra
iterations for the iterative numerical solution of this problem and
which gives rise to high-order convergence; and finally,
reference~\cite{britt2013high}, which successfully applies the method
of difference potentials to the variable-coefficient Zaremba problem,
with convergence order approximately equal to one.

The rest of this paper is organized as follows: after some necessary
preliminaries presented in Section~\ref{sec:preliminaries},
Section~\ref{sec:integral_equations} lays down the integral equation
system we use, and it studies the equivalence between the integral
equation system and the original Zaremba
problem~\eqref{eq:exterior}. Detailed asymptotics for the Zaremba
singularities at Dirichlet-Neumann junction are presented in
Section~\ref{sing_struct}. Building upon these results, further, and
exploiting an interesting connection of this problem with the
Fourier-Continuation method~\cite{bruno2009_1,albin2011}, Fourier
expansions for the integral densities are obtained in
Section~\ref{sec:smooth_geometries} which regularize all Zaremba
singularities.  Section~\ref{num_alg} then introduces a numerical
algorithm that incorporates the aforementioned Fourier Continuation
series on the basis of certain closed form integral expressions, and which
thus produces numerical solutions with high-order accuracy. A variety
of results in Section~\ref{num_res} demonstrates the quality of the
solutions and Zaremba eigenfunctions produced by the proposed
numerical approach, and Section~\ref{concl}, finally, presents a few
concluding remarks.

\section{Preliminaries}\label{sec:preliminaries}

We consider interior and exterior boundary value problems of the form
given in equation~\eqref{eq:exterior} for $u\in
H^1_\mathrm{loc}(\Omega)$ (with a Sommerfeld radiation condition in
case of exterior problems), where $\Omega \subset \mathbb{R}^2$
denotes either a bounded, open and simply-connected domain with a
smooth boundary (which we will generically call an {\em ``interior''}
domain) or the complement of the closure of such a domain (an {\em
  ``exterior''} domain), where the Dirichlet and Neumann boundary
portions $\Gamma_D$ and $\Gamma_N$ ($\Gamma=\Gamma_D\cup\Gamma_N$) are
disjoint relatively-open subsets of $\Gamma$ of positive measure
relative to $\Gamma$. Here, calling $B_R$ the ball of radius $R$
centered at the origin, we have denoted by
\begin{equation}
  H^1_\mathrm{loc}(\Omega) = \left\{ u : u\in H^1(\Omega \cap B_R)\hspace{1mm} \mbox{ for all }\hspace{1mm} R>0 \right\}
\end{equation}
the local Sobolev space of order one; of course
$H^1_\mathrm{loc}(\Omega)=H^1(\Omega)$ for bounded sets $\Omega$. The
Dirichlet and Neumann data $f$ and $g$ in~\eqref{eq:exterior}, in
turn, are elements of certain Sobolev spaces
(cf. Remark~\ref{rem:assumptions_f_g} item~\ref{ii}) which we define
in what follows. To do this we follow~\cite{mclean2000strongly} and we
first define, for a given relatively open subset $S\subseteq \Gamma$,
the space
\begin{equation*}
\widetilde{H}^{1/2}(S) = \{ \left. u \right| _{S} : \text{supp } u \subseteq S, u \in H^{1/2}(\Gamma) \},
\end{equation*}
The spaces associated with the  Dirichlet and Neumann data $f$ and $g$ are then defined by
\begin{equation*}
H^{1/2}(S) = \{ \left. u \right| _{S} : u \in H^{1/2}(\Gamma) \}.
\end{equation*}
and, using the prime notation $H'$ to denote the dual space of a given
Hilbert space $H$,
\begin{equation*}
  H^{-1/2}(S) = \left(\widetilde{H}^{1/2}(S)\right)',
\end{equation*}
respectively.

\begin{remark}
\label{rem:assumptions_f_g}
  Throughout this paper \dots
\begin{enumerate}[(i)]
\item \dots the term ``smooth'' is equated to ``infinitely
  differentiable '' and, as indicated above, it is assumed that the
  boundary of the domain $\Omega$ is smooth.
\item \label{ii} \dots it is assumed that the functions $f$ and $g$ in
  equation~\eqref{eq:exterior} are smooth. In particular, it follows
  that $f \in H^{1/2}(\Gamma_D)$ and $g \in H^{-1/2}(\Gamma_N)$---and,
  thus, that existence and uniqueness of the
  problem~\eqref{eq:exterior} hold~\cite{mclean2000strongly}.
\end{enumerate}
\end{remark}

 The boundary $\Gamma$ can be expressed in the form
\begin{equation}
\label{eq:gamma_decomposition}
\Gamma = \bigcup_{q=1}^{Q_N+Q_D} \Gamma_q 
\end{equation}
where $Q_D$ and $Q_N$ denote the numbers of smooth {\em connected}
Dirichlet and Neumann boundary portions, and where for $1\leq q \leq
Q_D$ (resp. for $Q_D+1\leq q \leq Q_D+Q_N$), $\Gamma_q$ denotes a
Dirichlet (resp. Neumann) portion of the boundary curve $\Gamma$ (see
e.g. Figure~\ref{fig:domain}). Clearly, letting
\begin{equation*}
J_D = \{ 1,\dots,Q_D\} \quad \mbox{and} \quad J_N = \{ Q_D+1,\dots,Q_D+ Q_N\}
\end{equation*}
we have that 
\begin{equation*}
\Gamma_D= \bigcup_{q \in J_D} \Gamma_q \quad  \mbox{and}  \quad \Gamma_N= \bigcup_{q \in J_N} \Gamma_q
\end{equation*}
are the subsets of $\Gamma$ upon which Dirichlet and Neumann boundary
conditions are enforced, respectively.  Note that consecutive values
of the index $q$ do not necessarily correspond to consecutive boundary
segments. Throughout this paper it is assumed, however, that no
Dirichlet-Dirichlet or Neumann-Neumann junctions occur, and, thus,
that every endpoint of a segment $\Gamma_q$ coincides with a
Dirichlet-Neumann junction. Clearly this is not a restriction:
consecutive Dirichlet (resp. Neumann) segments can be combined to
produce a partition which verifies the assumptions above.

\begin{remark}\label{rem:eigenvalue}
  In case $\Omega$ is an exterior domain, problem~\eqref{eq:exterior}
  admits unique solutions in $H^1_\mathrm{loc}(\Omega)$. On the
  other hand, if $\Omega$ is an interior domain, this problem is not
  well posed for a discrete set of real values $k_j$,
  $j=1,\dots,\infty$ ---the squares of which are the Zaremba
  eigenvalues, that is to say, the eigenvalues of the Laplace operator
  under the corresponding homogeneous mixed Dirichlet-Neumann
  (Zaremba) boundary conditions (see~\cite[Th. 4.10]{mclean2000strongly},~\cite{AkhBrunoNigam}).
\end{remark}

\section{Boundary integral equation formulation}
\label{sec:integral_equations}
In what follows we seek solutions of problem~\eqref{eq:exterior} on
the basis of the single-layer potential representation
\begin{equation}
\label{eq:single_layer}
u=\int _{\Gamma}G_{k}(x,y)\psi (y) ds_y,
\end{equation}
where $G_k(x,y)= \frac{i}{4} H^1_0(k|x-y|)$ is the Helmholtz Green
function in two-dimensional space. Taking into account well known
expressions~\cite[p. 40]{COLTON:1998} for the jump of the single layer
potential and its normal derivative across $\Gamma$, the boundary
conditions for the exterior (resp. interior) boundary value
problem~\eqref{eq:exterior} give rise to the integral equations
\begin{equation}
\begin{split}
\label{eq:exterior_integral_system}
\displaystyle \int _{\Gamma}G_{k}(x,y)\psi (y)ds_y&=f(x)\quad x\in\Gamma_D \quad \mbox{and} \\
\displaystyle \gamma \frac{\psi (x)}{2}+\int _{\Gamma}\frac{\partial
  G_{k}(x,y)}{\partial n_x}\psi (y)ds_y&=g(x)\quad x\in\Gamma _N
\end{split}
\end{equation}
with $\gamma = -1$ (resp. $\gamma = 1$).

Important properties of {\em both the interior and exterior} integral
equation problems~\eqref{eq:exterior_integral_system} relate to
existence of eigenvalues of certain {\em interior} problems for the
Laplace operator under either Dirichlet or Zaremba boundary
conditions. As shown in what follows, for example,  

\begin{enumerate}
\item\label{pt1} In case $\Omega$ is an exterior domain the integral
  equation system~\eqref{eq:exterior_integral_system} admits unique
  solutions if and only if $k^2$ is not a Dirichlet eigenvalue of the
  Laplace operator in $\R^2\setminus\Omega$.
\item \label{pt2} For such an exterior domain $\Omega$ the PDE
  problem~\eqref{eq:exterior} admits unique solutions for any real
  value of $k^2$ in spite of the lack of uniqueness implied in
  point~\ref{pt1} for certain wavenumbers $k$. A procedure is
  presented in Appendix~\ref{sec:uniqueness} which extends
  applicability of the proposed integral formulation to such values of
  $k$.
\item \label{pt3} In case $\Omega$ is an interior domain, in turn, the integral
  equation system is uniquely solvable provided $k^2$ is not a Zaremba
  eigenvalue of the Laplace operator in $\Omega$.
\item \label{pt4} The PDE problem~\eqref{eq:exterior} in such an interior domain
  $\Omega$ does not admit unique solutions, of course, for values of
  $k$ for which $k^2$ is a Zaremba eigenvalue in $\Omega$. In this
  case the eigenfunctions of the Zaremba Laplace operator can be
  expressed in terms of the representation
  formula~\eqref{eq:single_layer} for a certain density $\psi$ which
  satisfies~\eqref{eq:exterior_integral_system} with $f=0$ and $g=0$.
\end{enumerate}
A detailed treatment concerning points~\ref{pt1},~\ref{pt3}
and~\ref{pt4} above is presented in the remainder of this section
(Theorems~\ref{th:exterior},~\ref{th:interior} and
Definition~\ref{def:exterior_conjugate}). A corresponding discussion
concerning point~\ref{pt2}, in turn, is put forth in
Appendix~\ref{sec:uniqueness}.
\begin{definition}\label{def:exterior_conjugate}
  Given an interior (resp. exterior) domain $\Omega$ and a solution
  $u$ of~\eqref{eq:exterior} in $H^1_{loc}(\Omega)$, a function $w\in
  H^{1}_{loc}(\R^2 \setminus \Omega)$ is said to be a solution
  ``conjugate'' to $u$ if it satisfies
\begin{equation}
\label{eq:conjugate_problem}
\begin{split}
\Delta  w+k^2 w= 0\quad& x \in \R^2 \setminus \Omega  \\
w(x) = u(x)\quad& x\in\Gamma,
\end{split}
\end{equation}
as well as, in case $\R^2 \setminus \Omega$ is an exterior domain,
Sommerfeld's condition of radiation at infinity. Throughout this paper
the conjugate solution $w$ in case $\R^2 \setminus \Omega$ is an
exterior (resp. interior) domain will be denoted by $w = u_e$
(resp. $w=u_i$).
\end{definition}
\noindent

\begin{lemma}
\label{def_rem}
The conjugate solutions mentioned in Definition~\ref{def:exterior_conjugate} exist and are uniquely determined in
  each one of the following two cases:
\begin{enumerate}
\item \label{enum_1} $\R^2 \setminus\Omega$ is an exterior domain; and
\item \label{enum_2} $\R^2 \setminus\Omega$ is an interior domain
  and $k^2$ is not a Dirichlet eigenvalue of the Laplace operator in
  $\R^2 \setminus\Omega$.
\end{enumerate}
\end{lemma}
\begin{proof}
For both point~\ref{enum_1} and point~\ref{enum_2} we rely on the fact that the solution $u$ of the problem~\eqref{eq:exterior} is in $H^1_\mathrm{loc}(\Omega)$ (see Remark~\ref{rem:eigenvalue}), and, therefore, by the trace theorem (e.g.~\cite[Th 3.37]{mclean2000strongly}), its boundary values lie in $H^{1/2}(\Gamma)$.
For point~\ref{enum_1}  we then invoke~\cite[Th 9.11]{mclean2000strongly} to conclude that a uniquely determined conjugate solution $w\in H^1_\mathrm{loc}(\R^2\setminus\Omega)$ exists, as needed.
Point~\ref{enum_2} follows similarly using~\cite[Th 4.10]{mclean2000strongly} under the assumption that $k^2$ is not a Dirichlet eigenvalue in the interior
domain $\R^2 \setminus\Omega$. 
\end{proof}

\begin{theorem}
\label{th:exterior}
Let $\Omega$ be an exterior domain, and let $k\in\mathbb{R}$ be such
that $k^2$ is not a Dirichlet eigenvalue of the Laplace operator in the
interior domain $\R^2 \setminus\Omega$. Then the exterior integral
equation system~\eqref{eq:exterior_integral_system} ($\gamma=-1$, see also Remark~\ref{rem:assumptions_f_g} item~\ref{ii})
admits a unique solution given by
\begin{equation}
\label{eq:density_representation}
\displaystyle \psi = \frac{\partial  u_i}{\partial  n}-\frac{\partial  u}{\partial  n},
\end{equation}
where $u$ is the solution of the exterior mixed
problem~\eqref{eq:exterior}, and where $u_i$ is the (uniquely determined)
corresponding conjugate solution (Definition~\ref{def:exterior_conjugate} and Lemma~\ref{def_rem}).
\end{theorem}
\begin{proof}
\noindent
To obtain~\eqref{eq:density_representation} we first consider the
Green representation formula for the functions $u_i$ and $u$,
\begin{equation}\label{eq:representation_formula_u}
\begin{split}
 u_i(x)=\int _{\Gamma}\left(G_{k}(x,y)\frac{\partial  u_i}{\partial  n_y}-u_i\frac{\partial  G_{k}(x,y)}{\partial  n_y}\right)ds_y,  & \quad x\in\R^2 \setminus \Omega ,\\ 
 u(x)=\int _{\Gamma}\left(u\frac{\partial  G_{k}(x,y)}{\partial  n_y}-G_{k}(x,y)\frac{\partial  u}{\partial  n_y}\right)ds_y,  & \quad x\in\Omega,
\end{split}
\end{equation}
which, in view of the jump relations for the single and double layer
potential operators, in the limit $x\to \Gamma$ leads to the relations
\begin{equation}\label{eq:representation_formula_u_boundary}
\begin{split}
 \frac{u_i(x)}{2}=\int_{\Gamma}\left(G_{k}(x,y)\frac{\partial  u_i}{\partial n_y}-u_i\frac{\partial  G_{k}(x,y)}{\partial  n_y}\right)ds_y , \quad x\in\Gamma \\
\frac{u(x)}{2}=\int _{\Gamma}\left(u\frac{\partial  G_{k}(x,y)}{\partial n_y}-G_{k}(x,y)\frac{\partial  u}{\partial  n_y}\right)ds_y , \quad x\in\Gamma.
\end{split}
\end{equation}
Since for $x\in\Gamma _D$ we have $u(x)=u_i(x)=f(x)$, the sum of the
two equations in~\eqref{eq:representation_formula_u_boundary} yields
\begin{equation}
f(x)=\int _{\Gamma}G_{k}(x,y)\left(\frac{\partial  u_i}{\partial  n_y}-\frac{\partial  u}{\partial  n_y}\right)ds_y,\quad x\in\Gamma _D,
 \end{equation}
 and, thus, in view of~\eqref{eq:density_representation},
 \begin{equation}\label{eq:final_form_gamma_D}
f(x)=\int _{\Gamma}G_{k}(x,y)\psi (y) ds_y , \quad x\in\Gamma _D.
 \end{equation}
 Similarly, in the limit $x\to \Gamma$ the normal derivatives of the
 integrals in~\eqref{eq:representation_formula_u} give rise to the
 relations
\begin{equation}\label{normal_derivs}
\begin{split}
\displaystyle \frac{1}{2} \frac{\partial u_i}{\partial n_x}=\frac{\partial }{\partial  n_x}\int _{\Gamma}\left(G_{k}(x,y)\frac{\partial u_i}{\partial n_y}-u_i\frac{\partial  G_{k}(x,y)}{\partial  n_y}\right)ds_y , \quad x\in\Gamma \\
\displaystyle \frac{1}{2} \frac{\partial u}{\partial n_x}=\frac{\partial }{\partial  n_x}\int _{\Gamma}\left(u\frac{\partial  G_{k}(x,y)}{\partial  n_y}-G_{k}(x,y)\frac{\partial u}{\partial n_y}\right)ds_y,  \quad x\in\Gamma .
\end{split}
\end{equation}
The sum of the equations in~\eqref{normal_derivs} results in the identity
\begin{equation}
\frac{1}{2} \frac{\partial u_i}{\partial n_x}+\frac{1}{2} \frac{\partial u}{\partial n_x}=\int _{\Gamma }\frac{\partial G_{k}(x,y)}{\partial  n_x}\left(\frac{\partial  u_i}{\partial  n_y}-\frac{\partial  u}{\partial  n_y}\right)ds_y  \quad x\in\Gamma, 
\end{equation}
or, equivalently,
\begin{equation}
\label{eq:added_normal_derivative_exterior}
\frac{\partial u}{\partial n_x}=-\frac{1}{2}\left( \frac{\partial u_i}{\partial n_x}- \frac{\partial u}{\partial n_x}\right)+\int _{\Gamma }\frac{\partial
G_{k}(x,y)}{\partial  n_x}\left(\frac{\partial  u_i}{\partial  n_y}-\frac{\partial  u}{\partial  n_y}\right)ds_y  \quad x\in\Gamma.
\end{equation}
But for $x\in\Gamma_N$ we have $\displaystyle \frac{\partial
  u}{\partial n_x}=g(x)$, and, thus,
equation~\eqref{eq:added_normal_derivative_exterior} can be made to read
\begin{equation}\label{eq:final_form_gamma_N}
g(x)=-\frac{\psi(x)}{2}+\int _{\Gamma }\frac{\partial G_{k}(x,y)}{\partial n_x}\psi(y) ds_y,  \quad x\in\Gamma_N.
\end{equation}
Equations~\eqref{eq:final_form_gamma_D}
and~\eqref{eq:final_form_gamma_N} tell us that the density $\psi$ is a
solution of the exterior integral equation
system~\eqref{eq:exterior_integral_system}, as claimed.

In order to establish the solution uniqueness let $\xi$ be a solution
of equation~\eqref{eq:exterior_integral_system} with $f=0$ and $g=0$. Since as
mentioned above the exterior mixed problem is uniquely solvable, the
corresponding single layer potential
\begin{equation}
\displaystyle v=\int _{\Gamma}G_{k}(x,y)\xi(y)ds_y
\end{equation}
is equal to zero everywhere in $\Omega$. It then follows from the
continuity of the single layer potential that $v$ satisfies the
Dirichlet problem in the interior domain $\R^2 \setminus \Omega$ with
zero boundary values. Since by assumption $k^2$ is not a Dirichlet
eigenvalue of the Laplacian in $\R^2 \setminus \Omega$ it follows that
$v=0$ in that region as well. Thus, both the interior and exterior
normal derivatives vanish, and therefore so does their difference
$\xi$.  The proof is now complete.
\end{proof}

\begin{theorem}
\label{th:interior}
Let $\Omega$ be an interior domain. Then we have:
\begin{enumerate}
\item If $k^2$ is not a Zaremba eigenvalue (see
  Remark~\ref{rem:eigenvalue}), then the interior integral equation
  system~\eqref{eq:exterior_integral_system} ($\gamma$=1, see also
  Remark~\ref{rem:assumptions_f_g} item~\ref{ii}) admits a {\em
    unique} solution, which is given by
\begin{equation}
\label{eq:density_representation_interior}
\displaystyle \psi = \frac{\partial  u}{\partial  n}-\frac{\partial  u_e}{\partial  n}.
\end{equation}
Here $u$ is the solution of the interior mixed
problem~\eqref{eq:exterior}, and  $u_e$ is the solution
conjugate to $u$ (Definition~\ref{def:exterior_conjugate} and
Lemma~\ref{def_rem}).
\item If $k^2$ is a Zaremba eigenvalue, in turn, any
  eigenfunction $u$ satisfying~\eqref{eq:exterior} with $f=0$ and
  $g=0$ can be expressed by means of a single-layer representation 
\begin{equation}\label{eq:eigenfunction_single_layer}
u(x)=\int _{\Gamma}G_{k}(x,y)\left(\frac{\partial  u}{\partial  n_y}-\frac{\partial  u_e}{\partial  n_y}\right)ds_y ,\quad x\in \Omega\cup\Gamma.
\end{equation}
 where  $u_e$ denotes the
  conjugate solution corresponding to the eigenfunction $u$
  (Definition~\ref{def:exterior_conjugate} and Lemma~\ref{def_rem}).
\end{enumerate}
\end{theorem}

\begin{proof}
  We first consider properties that are common to Zaremba solutions
  and eigenfunctions, and which therefore relate to both points 1 and
  2 in the statement of the theorem. For any given solution $u$ of
  the interior mixed problem~\eqref{eq:exterior} ($u$ can be either
  the unique solution of the interior mixed problem in the case $k^2$
  is not an eigenvalue, or any eigenfunction satisfying
  ~\eqref{eq:exterior} with $f=0$ and $g=0$) the conjugate solution
  $u_e$ is uniquely defined (Lemma~\ref{def_rem}). Letting
\begin{equation}\label{eq:w(x)}
w(x) = \int _{\Gamma}G_{k}(x,y)\left(\frac{\partial  u}{\partial  n_y}-\frac{\partial  u_e}{\partial  n_y}\right)ds_y,
\end{equation}
using the Green representation formula for $u$ and $u_e$,
\begin{equation}\label{eq:representation_formula_ue}
\begin{split}
 u(x)=\int _{\Gamma}\left(G_{k}(x,y)\frac{\partial  u}{\partial  n_y}-u\frac{\partial  G_{k}(x,y)}{\partial  n_y}\right)ds_y,  & \quad x\in\Omega ,\\ 
 u_e(x)=\int _{\Gamma}\left(u_e\frac{\partial  G_{k}(x,y)}{\partial  n_y}-G_{k}(x,y)\frac{\partial  u_e}{\partial  n_y}\right)ds_y,  & \quad x\in\R^2 \setminus \Omega,
\end{split}
\end{equation}
and taking into account the jump relations for the double layer
potential as well as the fact that $u_e(x) = u(x)$ for $x\in\Gamma$
we obtain
\begin{equation}\label{eq:u=w}
u(x)=\int _{\Gamma}G_{k}(x,y)\left(\frac{\partial  u}{\partial  n_y}-\frac{\partial  u_e}{\partial  n_y}\right)ds_y = w(x),\quad x\in \Gamma.
\end{equation}
Similarly, taking normal derivatives of both sides of each equation
in~\eqref{eq:representation_formula_ue} at a point $x\in\Gamma$ we
obtain the equations
\begin{equation}\label{normal_derivs_ue}
\begin{split}
\displaystyle \frac{1}{2} \frac{\partial u}{\partial n_x}=\frac{\partial }{\partial  n_x}\int _{\Gamma}\left(G_{k}(x,y)\frac{\partial u}{\partial n_y}-u\frac{\partial  G_{k}(x,y)}{\partial  n_y}\right)ds_y , \quad x\in\Gamma , \\
\displaystyle \frac{1}{2} \frac{\partial u_e}{\partial n_x}=\frac{\partial }{\partial  n_x}\int _{\Gamma}\left(u_e\frac{\partial  G_{k}(x,y)}{\partial  n_y}-G_{k}(x,y)\frac{\partial u_e}{\partial n_y}\right)ds_y,  \quad x\in\Gamma ,
\end{split}
\end{equation}
whose sum yields
\begin{equation}
\label{eq:added_normal_derivative}
\frac{\partial u}{\partial n_x}=-\frac{1}{2}\left( \frac{\partial u}{\partial n_x}- \frac{\partial u_e}{\partial n_x}\right)+\int _{\Gamma }\frac{\partial
G_{k}(x,y)}{\partial  n_x}\left(\frac{\partial  u}{\partial  n_y}-\frac{\partial  u_e}{\partial  n_y}\right)ds_y = \frac{\partial  w}{\partial  n_x}  \quad x\in\Gamma.
\end{equation}

We now conclude the proof by applying these concepts to points 1 and
2 in the statement of the theorem.
\begin{enumerate}
\item In case $k^2$ is not an eigenvalue for the Laplace-Zaremba
  problem~\eqref{eq:exterior}, equations~\eqref{eq:u=w}
  and~\eqref{eq:added_normal_derivative} evaluated for $x\in\Gamma_D$
  and $x\in\Gamma_N$, respectively, show that the density $\psi$ given
  by~\eqref{eq:density_representation_interior} satisfies the integral
  equation system~\eqref{eq:exterior_integral_system} with $\gamma=1$,
  as desired.
\item In case $k^2$ is an eigenvalue for the Laplace-Zaremba
  problem~\eqref{eq:exterior}, in turn, let $u$ denote a corresponding
  eigenfunction. Equations~\eqref{eq:u=w}
  and~\eqref{eq:added_normal_derivative} along with the Green
  representation formula show that $u(x) = w(x)$ for all
  $x\in\Omega$. In other words,
  equation~\eqref{eq:eigenfunction_single_layer} is satisfied and the
  proof follows in this case as well.
\end{enumerate}
The proof is now complete.
\end{proof}

\section{Singularities of the solutions of equations~\eqref{eq:exterior} and~\eqref{eq:exterior_integral_system} }
\label{sing_struct}

\begin{figure}[h!]
\centering
\includegraphics[width=3in]{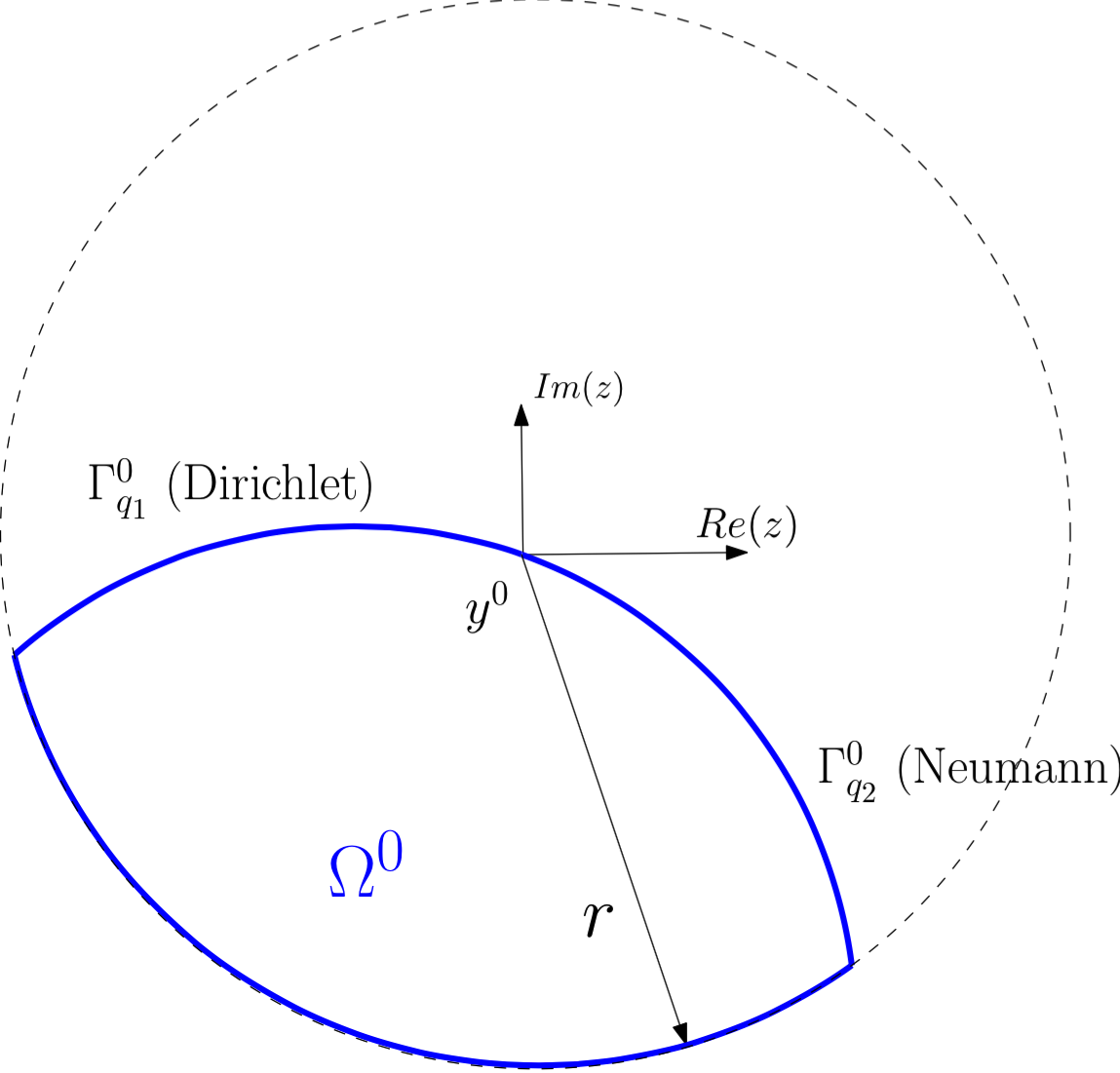}
\caption{Singular point $y^0$.}
\label{fig:singularity}
\end{figure}

With reference to equation~\eqref{eq:gamma_decomposition}, let $y^0 =
(y^0_1,y^0_2)\in\Gamma$ be a point which separates Dirichlet and
Neumann regions $\Gamma_{q_1}$ and $\Gamma_{q_2}$ ($q_1\in J_D$ and
$q_2 \in J_N$) within $\Gamma$. In order to express the singular
character around $y^0$ of both the solution $u(y)$ of
problem~\eqref{eq:exterior} ($y = (y_1,y_2) \in \Omega$) and the
corresponding integral equation density $\psi(y)$ in
equation~\eqref{eq:exterior_integral_system} ($y = (y_1,y_2) \in
\Gamma$) we consider the neighborhoods
\begin{equation}\label{eq:omega0}
\Omega^0 = \overline{\Omega \cap B(y^0,r)},\quad \Gamma^0_{q_1} =  \overline{\Gamma_{q_1}\cap B(y^0,r)} \quad\mbox{and}\quad \Gamma^0_{q_2} = \overline{\Gamma_{q_2}\cap B(y^0,r)}
\end{equation}
of the singular point $y^0$ relative to ${\Omega}$,
${\Gamma}_{q_1}$ and ${\Gamma}_{q_2}$,
respectively. Here for a set $A\subset\mathbb{R}^2$, $\overline{A}$
denotes the topological closure of $A$ in $\mathbb{R}^2$, $B(y^0,r)$
denotes the circle centered at $y^0$ of radius $r$, and $r>0$ is
sufficiently small that $B(y^0,r)$ only has nonempty intersections
with $\overline{\Gamma}_{q}$ for the Dirichlet index $q=q_1$ and the
Neumann index $q=q_2$. Additionally, we use certain functions
$\widehat{u}_{y^0} = \widehat{u}_{y^0}(z)$, $\widehat{\psi}_{y^0}^1 =
\widehat{\psi}_{y^0}^1(d)$ and $\widehat{\psi}_{y^0}^2 =
\widehat{\psi}_{y^0}^2(d)$ where the Dirichlet (resp. Neumann)
function $\widehat{\psi}_{y^0}^1$ (resp $\widehat{\psi}_{y^0}^2$) is
the density as a function of the distance $d$ to the point $y^0$ in
$\Gamma^0_{q_1}$ (resp. $ \Gamma^0_{q_2}$), and where $z =(y_1-y^0_1)
+ i (y_2-y^0_2)$ is a complex variable (see
Figure~\ref{fig:singularity}). The functions $\widehat{u}_{y^0}$,
$\widehat{\psi}_{y^0}^1$ and $\widehat{\psi}_{y^0}^2$ are given by
\begin{equation}
\label{eq:real_to_complex}
\begin{split}
  &\widehat{u}_{y^0}(z) = u(y),  \\
  &\psi(y) =\widehat{\psi}_{y^0}^1(d(y)) \quad y \in \Gamma_{q_1}^0,\\
  &\psi(y) =\widehat{\psi}_{y^0}^2 (d(y)) \quad y \in \Gamma_{q_2}^0, 
\end{split}
\end{equation}
where, as mentioned above
\begin{equation}\label{local}
z = (y_1-y^0_1) + i (y_2-y^0_2) \quad \mbox{and}\quad d(y) = \sqrt{(y_1-y^0_1)^2+(y_2-y^0_2)^2}.
\end{equation}

It is known~\cite{wigley1964,wigley1970mixed} that, under our
assumption that the curve $\Gamma$ is globally smooth, for any given
integer $\mathcal{N}$ the function $\widehat{u}_{y^0}(z)$ can be
expressed in the form
\begin{equation}
\label{eq:u_asymptotics}
  \widehat{u}_{y^0}(z)= \log(z) P_{y^0}^{1,\mathcal{N}} + \log(\overline{z})P_{y^0}^{2,\mathcal{N}} + P_{y^0}^{3,\mathcal{N}} + o(z^{\mathcal{N}})
\end{equation} 
for all $z$ in a neighborhood of the origin, where
$P_{y^0}^{1,\mathcal{N}},P_{y^0}^{2,\mathcal{N}}$ and
$P_{y^0}^{3,\mathcal{N}}$ are $\mathcal{N}$-dependent polynomial
functions of $z$, $\overline{z}$, $z^{1/2}$, $\overline{z}^{1/2}$, $z\log(z)$, $\overline{z} \log(\overline{z})$.  

\begin{remark}\label{remark:polynomial_powers}
In the asymptotic expansion~\eqref{eq:u_asymptotics}, and, indeed, in all similar asymptotic expansions in this paper, it is assumed that none of the right hand side polynomials contain terms that, multiplied by the relevant factors, could be included in the error term. 
\end{remark}

Under our
standing assumption of smoothness of the domain boundary the following 
two theorems provide 1)~Finer details on the
asymptotics~\eqref{eq:u_asymptotics} as well as 2)~A corresponding
asymptotic expression around $y^0$ for the solutions
$\widehat{\psi}_{y^0}^1(d)$ and $\widehat{\psi}_{y^0}^2(d)$ of the
integral-equation system~\eqref{eq:exterior_integral_system}.
\begin{theorem} \label{th:solution_singularity} Let $y^0$ be a
  Dirichlet-Neumann point as described above in this section. Then,
  given an arbitrary integer $\mathcal{N}$, the function
  $\widehat{u}_{y^0}(z)$ can be expressed in the form
\begin{equation}
\label{final_form}
\widehat{u}_{y^0}(z) = P^\mathcal{N}_{y^0}(z^{1/2},\overline{z}^{1/2}) + o(z^{\mathcal{N}})
\end{equation}
around $y^0$, where $P^\mathcal{N}_{y^0}$ is an
$\mathcal{N}$-dependent polynomial function of its arguments; see Remark~\ref{remark:polynomial_powers}.
\end{theorem}

\begin{theorem}\label{th:density_singularity}
  Let $y^0$ be a Dirichlet-Neumann point. Then given an arbitrary
  integer $\mathcal{N}$ the functions $\widehat{\psi}_{y^0}^1(d)$ and
  $\widehat{\psi}_{y^0}^2(d)$ can be expressed in the forms
\begin{equation}
\label{eq:final_form_density}
\begin{split}
  &\widehat{\psi}_{y^0}^1(d) = d^{-1/2} Q_{y^0}^{1,\mathcal{N}}(d^{1/2}) + o(d^{\mathcal{N}-1})\quad\mbox{and}  \\
  &\widehat{\psi}_{y^0}^2(d) = d^{-1/2}
  Q_{y^0}^{2,\mathcal{N}}(d^{1/2}) + o(d^{\mathcal{N}-1})
\end{split}
\end{equation}
around $d=0$, where $Q_{y^0}^{1,\mathcal{N}}$ and
$Q_{y^0}^{2,\mathcal{N}}$ are $\mathcal{N}$-dependent polynomial
functions of their arguments; see Remark~\ref{remark:polynomial_powers}.
\end{theorem}

Note in particular that Theorem~\ref{th:solution_singularity} shows
that, under our assumptions of boundary smoothness, all logarithmic
terms in equation~\eqref{eq:u_asymptotics} actually drop out. The
proofs of these theorems (which are given in
Sections~\ref{section_singularities}
and~\ref{sec:density_singularities}, respectively) utilize a certain
conformal map introduced in Section~\ref{sec:conformal_mapping} that
transforms $\Omega^0$ into a semicircular region.

\subsection{Conformal mapping}
\label{sec:conformal_mapping}
Following~\cite{wigley1964}, in order to establish
Theorem~\ref{th:solution_singularity} we identify $\mathbb{R}^2$ with
the complex plane $\mathbb{C}$ via the aforementioned relationship $z
= (y_1-y^0_1)+i(y_2-y^0_2)\leftrightarrow (y_1-y^0_1,y_2-y^0_2)$,
and we utilize a conformal map $z = w(\xi)$ which maps the semi-circular
region $D_A = \left\{ \xi\in \mathbb{C}: |\xi|\le A\,\mbox{ and }\,
  \textrm{Im}(\xi)\le0 \right \}$ ) in the complex $\xi$-plane
(Figure~\ref{fig:semicircle}) onto the domain $\Omega^0$
(equation~\eqref{eq:omega0}) in the complex $z-$plane. We assume, as
we may, that $w$ maps the origin to itself and that the intervals
$\{\textrm{Im}(\xi)=0, 0 \le \textrm{Re}(\xi) \le A\}$ and
$\{\textrm{Im}(\xi)=0, -A \le \textrm{Re}(\xi) \le 0\}$ are mapped
onto the boundary segments $ \Gamma^0_{q_1}$ and $ \Gamma^0_{q_2}$,
respectively.
\begin{figure}[h!]
\centering
\includegraphics[width=6in]{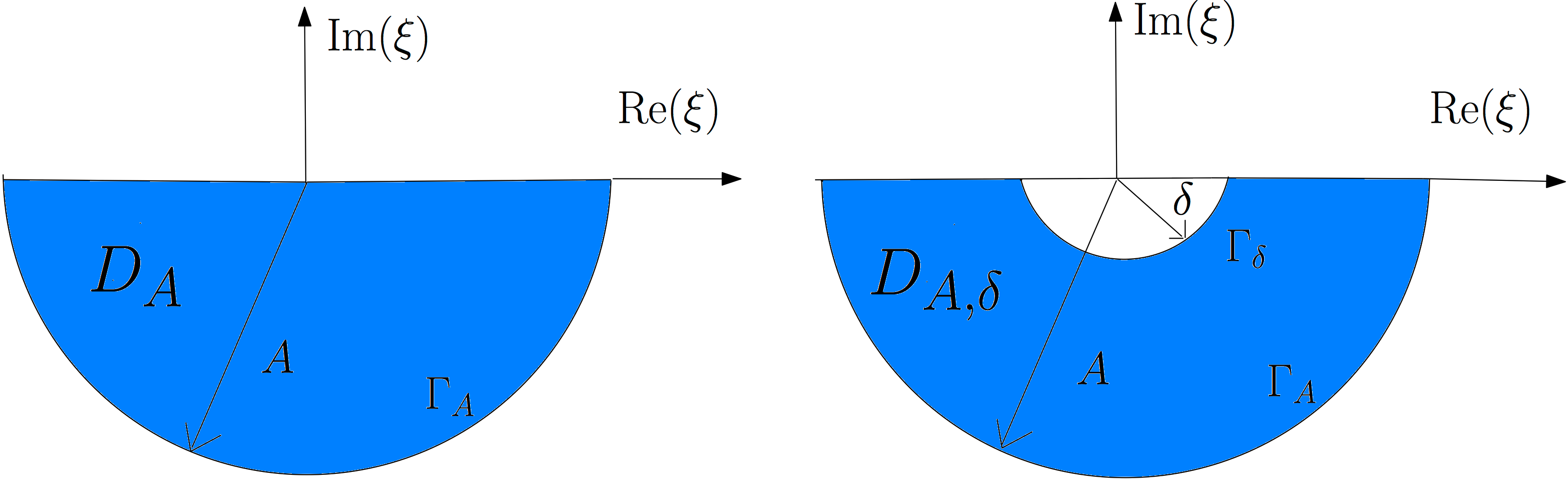}
\caption{Semi-circular and semi-annular Green-identity regions.}
\label{fig:semicircle}
\end{figure}

Letting 
\begin{equation}\label{U-u}
  U(\xi) = \widehat{u}_{y^0}(w(\xi))
\end{equation}
we note that, in view of the relation  $\Delta_{\xi} U(\xi) = \Delta_z \widehat{u}_{y^0}(w(\xi))\cdot|w'(\xi))|^2$ satisfied by a complex analytic function $w$ (see~\cite[eq. 5.4.17]{fokas2003complex}), $U$ satisfies the second order elliptic equation
\begin{align}
\label{eq:semicircle_pde}
&\Delta  U+ K(\xi) U =0 &\quad&\mbox{for}\quad \xi\in \textrm{int}(D_A),\\
&U(\xi)=F(\xi)&\quad&\mbox{for}\quad \textrm{Im}(\xi)=0,\textrm{Re}(\xi) > 0, \\
&\displaystyle \frac{\partial  U(\xi)}{\partial  n}=G(\xi) &\quad&\mbox{for}\quad \textrm{Im}(\xi)=0,\textrm{Re}(\xi) < 0,\quad\mbox{and} \\
&U(\xi)=M(\xi)&\quad&\mbox{for}\quad |\xi| =A. \label{eq:semicircle_pde1}
\end{align}
Here $F(\xi) = f(w(\xi))$, $G(\xi) = g(w(\xi))$ and $M(\xi)=
u(w(\xi))$. (The function $M$ is thus obtained from the restriction of
the solution $u$ to the set $\partial\Omega^0 \setminus
(\Gamma^0_{q_1}\cup \Gamma^0_{q_2})$; see equation~\eqref{eq:omega0}
and Figure~\ref{fig:singularity}).


\subsection{Proof of theorem~\ref{th:solution_singularity}}
\label{section_singularities} 
The proof of this theorem, which, under the present scope of
smooth-domain problems establishes a result stronger
than~\cite[Th. 3.2]{wigley1964}, does incorporate some of the lines of
the proof provided in that reference. In what follows we use the
Laplace-Zaremba Green function
\begin{equation}
\label{greensfunction}
H(t,\xi)=-\frac{1}{2 \pi }\left\{ \log  | t-\xi| + \log |t- \overline{\xi}|-2 \log \left|\sqrt{t}+\sqrt{\xi}\right| -2 \log \left|\sqrt{t}-\sqrt{\overline{\xi}}\ \right| \right\}
\end{equation}
for the lower half plane with homogeneous Dirichlet (resp. Neumann)
boundary conditions on the positive (resp.negative) real axis in terms
of the complex variables $t = t_1 + it_2 = (t_1,t_2)$ and $\xi = \xi_1
+ i\xi_2 = (\xi_1,\xi_2)$. The branches of the square roots
in~\eqref{greensfunction} are given by $\sqrt{t} = \sqrt{\rho_t e^{i
    \theta_t}} = \sqrt{\rho_t} e^{i \theta_t/2}$ and $\sqrt{\xi} =
\sqrt{\rho_\xi e^{i \theta_\xi}} = \sqrt{\rho_\xi} e^{i \theta_\xi/2}$
where $(\rho_t ,\theta_t)$ and $(\rho_\xi,\theta_\xi)$ denote polar
coordinates in the complex $t$- and $\xi$-plane respectively
($-\pi\leq \theta_t ,\theta_\xi < \pi$). Note that, with these
conventions the domain $D_A$ in the $t$ variables is given by $\rho_t
\leq A$ and $-\pi \le \theta_t \le 0$. The following Lemma establishes
certain important properties of the aforementioned Green function.
\begin{lemma}\label{like_folland}
  The function $H=H(t,\xi)$ (equation~\eqref{greensfunction}) is
  indeed a Laplace-Zaremba Green function for the lower half plane
  with a Dirichlet-Neumann junction at the origin, that is we have
  $\Delta_tH(t,\xi) = -\delta(t-\xi)$ and $H$ satisfies
\begin{equation}
\label{greensfunction_zero}
H(t,\xi) = 0\quad \mbox{for}\quad\theta_t=0\quad\mbox{and}\quad
\frac{\partial  H}{\partial n_t }(t,\xi) =  0\quad \mbox{for}\quad \theta_t =-\pi.
\end{equation}
In addition, for a certain constant $C$ we have
\begin{equation}\label{eq:folland}
\int^A_0 \left| \frac{\partial }{\partial  n_t} \left(  H(t,\xi) \right) \right|dt  \le C
\end{equation}
for all $\xi\in \mathbb{C}$.
\end{lemma}
\begin{proof}
  The function $H(t,\xi)$ can be re-expressed in the form
\begin{equation}\label{greensfunction_reformulated}
H(t,\xi)=-\frac{1}{2 \pi }\log \frac{|\sqrt{t}-\sqrt{\xi}| |\sqrt{t}+ \sqrt{\overline{\xi}}| }{| \sqrt{t} + \sqrt{\xi}| |\sqrt{t}-\sqrt{\overline{\xi}}|}.
\end{equation}
The first statement in equation~\eqref{greensfunction_zero} follows
from the relations $|\sqrt{t}-\sqrt{\xi}| =
|\sqrt{t}-\sqrt{\overline{\xi}}|$ and $|\sqrt{t}+
\sqrt{\overline{\xi}}| = | \sqrt{t} + \sqrt{\xi}|$, which hold for
$\theta_t=0$ ($t>0$) since, in view of our selection of branch cuts we
have
\begin{equation}
\label{sqrt_xi}
\sqrt{\overline{\xi}} = \overline{\sqrt{\xi}}.
\end{equation}
In order to establish the second statement in~\eqref{greensfunction_zero} and
equation~\eqref{eq:folland} we consider the relations
\begin{equation}\label{log_normal_derivative}
\begin{split}
  \frac{\partial }{\partial t_2} \log|\sqrt{t} - (z_1 + i z_2)| =\frac{z_1}{2 \sqrt{-t_1} \left(z_1^2+(\sqrt{-t_1} + z_2)^2 \right)}&\quad \mbox{for}\quad \theta_t =-\pi,\quad\mbox{and} \\
  \frac{\partial }{\partial t_2} \log|\sqrt{t} - (z_1 + i
  z_2)|=-\frac{z_2}{2 \sqrt{t_1} \left((z_1-\sqrt{t_1})^2+z_2^2
    \right)} &\quad \mbox{for}\quad \theta_t = 0,
\end{split}
\end{equation}
which are valid for all complex numbers $z=z_1+iz_2$, and we note
that, on the axis $t_2=0$ we have $\frac{\partial }{\partial n_t} =
\frac{\partial }{\partial t_2}$. Thus, the second statement
in~\eqref{greensfunction_zero} follows from application of the first
equation in~\eqref{log_normal_derivative} to each one of the four
logarithmic terms that result from expansion of
equation~\eqref{greensfunction_reformulated}. In order to establish a
bound of the form~\eqref{eq:folland}, finally, we use the second
equation in~\eqref{log_normal_derivative} to obtain the expression
\begin{equation}
\left| \frac{\partial }{\partial  n_t} \left(  H(t,\xi) \right) \right| = \frac{|\textrm{Im}(\sqrt{\xi})|}{2 \pi \sqrt{t} } \left( \frac{1}{(\sqrt{t} + \textrm{Re}(\sqrt{\xi}))^2 + \textrm{Im}(\sqrt{\xi})^2} + \frac{1}{(\sqrt{t} - \textrm{Re}(\sqrt{\xi}))^2 + \textrm{Im}(\sqrt{\xi})^2}\right)
\end{equation}
for the absolute value of the derivative
of~\eqref{greensfunction_reformulated} with respect to $t_2$.  It is
then easy to check that
\begin{equation}\label{H_normal_derivative_integral}
\int^A_0 \left| \frac{\partial }{\partial  n_t} \left(  H(t,\xi) \right) \right|dt = \frac{|\textrm{Im}(\sqrt{\xi})|}{\pi \textrm{Im}(\sqrt{\xi}) } \left(\arctan\frac{\sqrt{A} - \textrm{Re}(\sqrt{\xi})}{\textrm{Im}(\sqrt{\xi})} + \arctan\frac{\sqrt{A} + \textrm{Re}(\sqrt{\xi})}{\textrm{Im}(\sqrt{\xi})} \right)
\end{equation}
for all $\xi \in \mathbb{C}\setminus[0;A]$. Since the right-hand side
of equation~\eqref{H_normal_derivative_integral} is uniformly bounded
for $\xi \in \mathbb{C}\setminus[0;A]$ and since for $\xi \in [0;A]$
the integrand in~\eqref{eq:folland} vanishes (in view of the second
expression in~\eqref{log_normal_derivative}), we see that there exists
a constant $C$ such that equation~\eqref{eq:folland} holds for all
$\xi \in \mathbb{C}$, as needed, and the proof is thus complete.
\end{proof}

The proof of the theorem~\ref{th:solution_singularity} is based on a
bootstrapping argument which is initiated by the simple but suboptimal
asymptotic relation put forth in the following lemma
\begin{lemma}\label{lemma:u_xi_mu}
The solution $U$ of
the problem~\eqref{eq:semicircle_pde}--\eqref{eq:semicircle_pde1} satisfies the asymptotic relation
\begin{equation}
\label{eq:u_xi_mu}
U(\xi)=o\left(\xi^{\mu }\right)
\end{equation}
for all $-\frac{1}{2}<\mu <0$.
\end{lemma}
\begin{proof}
  To establish this relation we consider the Green formula
\begin{equation}
\label{eq:green_formula_U_DA}
U(\xi)=\int \int _{D_A}H(t,\xi) \Delta  U(t) d x_t d y_t+\int _{\partial D_A}\left\{ U(t) \frac{\partial }{\partial  n_t}H(t,\xi)-H(t,\xi)
\frac{\partial }{\partial  n_t}U(t)\right\}ds_t.
\end{equation}
Since $H$ satisfies~\eqref{greensfunction_zero} as it befits a Green
function for~\eqref{eq:semicircle_pde}--\eqref{eq:semicircle_pde1},
denoting by $\Gamma_A$ the radius-$A$ part of $\partial D_A$ it
follows that
\begin{equation}
\begin{split}
\label{U_bound}
|U(\xi)| \le &\left| \int \int _{D_A} H(t,\xi) K(t) U(t) d x_t d y_t \right|  + \left| \int^A_0 F(x) \frac{\partial }{\partial  n_t}H(t,\xi)dt \right|\\
 +& \left| \int ^A_0 G(-t)H(-t ,\xi)dt\right| + \left|\int _{\Gamma_A}\left\{ U(t) \frac{\partial }{\partial  n_t}H(t,\xi)-H(t,\xi)
\frac{\partial }{\partial  n_t}U(t)\right\}ds_t \right| .
\end{split}
\end{equation}

For the integral over the outer arc $\Gamma_A$ in~\eqref{U_bound} we have
 \begin{equation}
\left|\int _{\Gamma_A}\left\{ U(t) \frac{\partial }{\partial  n_t}H(t,\xi)-H(t,\xi)
\frac{\partial }{\partial  n_t}U(t)\right\}ds_t \right| \le C,\quad\mbox{for}\quad |\xi| < A/2,
\end{equation}
as it can be checked easily in view of the boundedness of the
integrands for $\xi$ near the origin.  Taking into account that $u \in
H^1_\mathrm{loc}(\Omega)$ (see Remark~\ref{rem:eigenvalue}), on the
other hand, it easily follows that $U\in H^1(D_A)$, and thus, bounding
the absolute value of the first integral in~\eqref{U_bound}
by means of the Cauchy-Schwarz inequality, for $\xi$ near $0$ we
obtain the uniform estimate
\begin{equation}
\left| \int \int _{D_A} H(t,\xi) K(t) U(t) d x_t d y_t \right| \le ||H||_{L^2} ||U||_{L^2} \max(K).
\end{equation}

In order to estimate the second and third integrals in~\eqref{U_bound}
we note that the integrals of the absolute values of $H$ and $\partial
H/\partial n_t$ are uniformly bounded, as can be checked easily for
the former, and as it is established in Lemma~\ref{like_folland} of
the latter. The boundedness of $F$ and $G$ (which are smooth functions
in view of Remark~\ref{rem:assumptions_f_g}) thus implies the uniform
boundedness of the function $U$ near the origin. The
relation~\eqref{eq:u_xi_mu} therefore follows for all $-\frac{1}{2}<\mu <0$,
and the proof is complete.
\end{proof}

\begin{corollary}\label{cor1}
The derivatives of the solution $U$ of
the problem~\eqref{eq:semicircle_pde}--\eqref{eq:semicircle_pde1} satisfy the asymptotic relation
\begin{equation}
\label{eq:du_xi_mu}
  D^h U=o\left(\xi^{\mu -h}\right)
\end{equation}
for all $-\frac{1}{2}<\mu <0$.
\end{corollary}

\begin{proof}
See~\cite[Section 4]{wigley1964}.
\end{proof}

A key element in the bootstrapping algorithm mentioned at the
beginning of this section is a representation formula for the function
$U$ that is presented in the following Lemma. 
\begin{lemma}
The solution $U$ of
equation~\eqref{eq:semicircle_pde}--\eqref{eq:semicircle_pde1} admits the representation
\begin{equation}
\label{final_representation}
\begin{split}
U(\xi)&=-\frac{1}{2\pi}\{ \Lambda_1 \left(-K(t) U\left(t\right),\xi,1\right)+\Lambda_1 \left(-K(t) U\left(t\right),\overline{\xi},1 \right)-2\Lambda_1\left(-K(t) U\left(t\right),\xi ,2\right)\\
&-2\Lambda _1\left(- K(t) U\left(t\right),\overline{\xi},2\right)- \Lambda_3(F(t),\xi,1) -\Lambda_3\left(F(t),\overline{\xi},1\right)+2\Lambda _3\left(F(t),\xi ,2\right)\\
&+2\Lambda _3\left(F(t),\overline{\xi},2\right)-\Lambda _2\left(G\left(-t\right),-\xi,1\right)-\Lambda _2\left(G\left(-t\right),-\overline{\xi},1\right)+2\Lambda _2\left(G\left(-t\right),-\xi,2\right)\\
&+2\Lambda_2\left(G\left(-t\right),- \overline{\xi},2\right)\}+p_1\left(\sqrt{\xi}\right)+p_2\left(\sqrt{\overline{\xi}}\right),
\end{split}
\end{equation}
where
\begin{equation}
\label{eq:lambda_functions}
\begin{split}
\Lambda_1(q(t),\xi,\eta )&:=\int_{-\pi }^0 \int_0^A q(t)\log \left|t^{\frac{1}{\eta
}} -\xi^{\frac{1}{\eta }}\right|\rho_t d\rho_t d\theta_t,\\
\Lambda_2(q(t),\xi,\eta )&:=\int_0^A q(t)\log \left|t^{\frac{1}{\eta }}-\xi^{\frac{1}{\eta }}\right| dt,\\
\Lambda_3(q(t),\xi,\eta )&:=\int ^A_0 q(t)\frac{1 }{t}\frac{\partial }{\partial  \theta_t}\log \left|t^{\frac{1}{\eta }}-\xi^{\frac{1}{\eta }}\right| dt,
\end{split}
\end{equation}
and where $p_1$ and $p_2$ denote power series with positive radii of
convergence.
\end{lemma}
\begin{proof}
  Applying the Green formula on the set 
\begin{equation}\label{D_a_delta}
D_{A,\delta}= \left\{ \xi\in
    \mathbb{C}: \delta\le |\xi|\le A\,\mbox{ and }\,
    \textrm{Im}(\xi)\le0 \right \}
\end{equation}
 (right portion of
  Figure~\ref{fig:semicircle}) we obtain the expression
\begin{equation}
\label{eq:green_formula_U}
U(\xi)=\int \int _{D_{A,\delta}}H(t,\xi) \Delta  U(t) d x_t d y_t+\int _{\partial D_{A,\delta}}\left\{ U(t) \frac{\partial }{\partial  n_t}H(t,\xi)-H(t,\xi)
\frac{\partial }{\partial  n_t}U(t)\right\}ds_t.
\end{equation}
Further, for fixed $\xi\ne 0$ we have
\begin{equation}
\label{eq:H_asymptotics}
H(t,\xi)=\mathcal{O}\left(\sqrt{t}\right)\text{ as } t \to 0,
\end{equation}
and therefore, in view of Lemma~\ref{lemma:u_xi_mu} and~\eqref{eq:du_xi_mu},
\begin{equation}
\label{greensfunction_gotozero}
\begin{split}
U(t)& H(t,\xi) =o\left(t^{\mu }\right), \\
U(t) \frac{\partial }{\partial  \rho_t }H(t,\xi) &- H(t,\xi) \frac{\partial }{\partial  \rho_t }U(t)=o\left(t^{\mu -\frac{1}{2}}\right).
\end{split}
\end{equation}
Letting $\Gamma _{\delta }$ denote the radius-$\delta$ arc within the
boundary $\partial D_{A,\delta}$ of $D_{A,\delta}$
(Figure~\ref{fig:semicircle}), and noting that for
$t\in\Gamma_{\delta}$ we have $\frac{\partial }{\partial \rho_t } =
\frac{\partial }{\partial n_t }$, in view
of~\eqref{greensfunction_gotozero} we obtain
\begin{equation}
\label{simplify_arc1}
\int _{\Gamma _{\delta }}\left\{ U(t) \frac{\partial }{\partial  n_t}H(t,\xi)-H(t,\xi) \frac{\partial }{\partial  n_t}U(t)\right\}ds_{t} \to  0 \text{ as } \delta \to  0.
\end{equation}
Further, exploiting
the fact that the Green function~\eqref{greensfunction} is a jointly
analytic function of $\sqrt{\xi}$ and $\sqrt{\overline{\xi}}$ for
$t\in\Gamma_A$ and $\xi$ around $\xi=0$, we obtain
\begin{equation}
\label{simplify_arc2}
\int _{\Gamma _A}\left\{ U(t) \frac{\partial }{\partial  n_t}H(t,\xi)-H(t,\xi) \frac{\partial }{\partial  n_t}U(t)\right\}ds_t=p_1\left(\sqrt{\xi}\right)+p_2\left(\sqrt{\overline{\xi}}\right),
\end{equation}
where $p_1$ and $p_2$ denote power series with positive radii of
convergence.

Letting $\delta\to 0$ in~\eqref{eq:green_formula_U}
and using~\eqref{simplify_arc1} and~\eqref{simplify_arc2} we finally
obtain
\begin{equation}
\begin{split}
\label{representation_simplified}
U(\xi)=\int^0_{-\pi }\int ^A_0 &H(t,\xi) (-K(t) U(t)) \rho_t d \rho_t  d\theta_t - \int^A_0 F(x) \frac{1 }{t}\frac{\partial }{\partial  \theta_t}H(t,\xi)dt\\
-&\int ^A_0 G(-t)H(-t ,\xi)dt+p_1\left(\sqrt{\xi}\right)+p_2\left(\sqrt{\overline{\xi}}\right).
\end{split}
\end{equation}
Using the definitions~\eqref{eq:lambda_functions} for the functions
$\Lambda_1$, $\Lambda_2$ and $\Lambda_3$,
equation~\eqref{representation_simplified} is equivalent to
equation~\eqref{final_representation} and the proof is complete.
\end{proof}

In order to determine the singular character of $U(\xi)$ around
$\xi=0$ (and therefore that of $\widehat{u}_{y^0}(z)$ around $z=0$) we
study the corresponding asymptotic behavior of each one of the
$\Lambda$-terms in equation~\eqref{final_representation}. An important
part of this discussion is the following Lemma, which presents certain
regularity properties of the operators $\Lambda_1$, $\Lambda_2$ and
$\Lambda_3$.
\begin{lemma}
\label{Lemma 3}
Let $\alpha \ge 0$, $\beta>-1$ and $\gamma>-1$, and let $\eta =1$ or
$\eta =2$. Then
\begin{equation}
\label{eq:lambda_1}
 \Lambda _1\left(t^{\beta }\overline{t}^{\gamma },\xi,\eta \right)=C_1 \xi^{\beta+\gamma +2}+C_2 \overline{\xi}^{\beta+\gamma +2}+C_3 \xi^{\beta+1}\overline{\xi}^{\gamma +1}+p_1\left(\xi^{\frac{1}{\eta
}}\right)+p_2\left(\overline{\xi}^{\frac{1}{\eta }}\right),
\end{equation}
\begin{equation}
\label{eq:lambda_2}
  \Lambda_2(t^\beta,\xi,\eta) = C_4 \xi^{\beta+1} + C_5 \overline{\xi}^{\beta+1} + p_3(\xi^{1/\eta}) + p_4(\overline{\xi}^{1/\eta})\quad\mbox{and}
\end{equation}
\begin{equation}
\label{eq:lambda_3}
  \Lambda_3(t^{\alpha}, \xi, \eta) = C_6 \xi^{\alpha} +C_7 \overline{\xi}^{\alpha} +p_5(\xi^{1/\eta})+p_6(\overline{\xi}^{1/\eta}).
\end{equation}
For general functions $g(t)\in C^\ell(D_A)$ and $h(t)\in
C^\ell((0,A])$ satisfying $g(t) = o(t^\gamma)$ and $h(t) =
o(t^\alpha)$ as $t\to 0$, further, we have
\begin{equation}
\label{eq:lambda_1_rem}
 \Lambda _1\left(g(t),\xi,\eta \right)=q_1\left(\xi^{\frac{1}{\eta}}\right)+q_2\left(\overline{\xi}^{\frac{1}{\eta }}\right) + o(\xi^{\gamma+2}),
\end{equation}
\begin{equation}
\label{eq:lambda_2_rem}
  \Lambda _2\left(g(t),\xi,\eta \right)=q_3\left(\xi^{\frac{1}{\eta}}\right)+q_4\left(\overline{\xi}^{\frac{1}{\eta }}\right) + o(\xi^{\gamma+1})\quad\mbox{and}
\end{equation}
\begin{equation}
\label{eq:lambda_3_rem}
 \Lambda _3\left(h(t),\xi,\eta \right)=q_5\left(\xi^{\frac{1}{\eta}}\right)+q_6\left(\overline{\xi}^{\frac{1}{\eta }}\right) + o(\xi^{\alpha}),
\end{equation}
(see Remark~\ref{remark:polynomial_powers}). Here $p_i$ (resp. $q_i$), $ i=1,\dots, 6$, are power series with
positive radii of convergence (resp. polynomials), $C_i$, $i=1\dots7$
are complex constants, and the expansions are $\ell$-times
differentiable as $\xi \to 0$---in the sense of Wigley: the
derivatives of the left hand sides in~\eqref{eq:lambda_1_rem} through~\eqref{eq:lambda_3_rem} are
equal to the corresponding derivatives of the first two term of the
right hand sides, with error terms given by the ``formal'' derivatives
of the corresponding error terms---e.g. $d/d\xi\left( o(\xi^\alpha)\right) =
o(\xi^{(\alpha-1)})$.
\end{lemma}
\begin{proof}[Proof of Lemma~\ref{Lemma 3}]
The proof follows by specializing the proofs of Lemmas $7.1-7.2$ and $8.1-8.6$ in~\cite{wigley1964}.
\end{proof}

We are now ready to provide the main proof of this section.
\begin{proof} [Proof of Theorem~\ref{th:solution_singularity}]
  Since the solutions $\widehat{u}_{y^0}$ and $U$ are related by
  equation~\eqref{U-u}, using the classical
  result~\cite[Th. IV]{warschawski1935higher} (which establishes, in
  particular, that the conformal mapping $w$ is smooth up to and
  including the boundary for any smooth segment of the domain
  boundary; see also~\cite{warschawski1961differentiability}) and
  expanding $w(\xi)$ in Taylor series around $\xi=0$, we see that it
  suffices to prove that for an arbitrary integer $\mathcal{M}$ we
  have the asymptotic relation
\begin{equation}
\label{final_form_U}
U(\xi) = \mathcal{Q}^\mathcal{M}(\xi^{1/2},\overline{\xi}^{1/2}) + o(\xi^{\mathcal{M}})
\end{equation}
for the solution $U(\xi)$ of the
problem~\eqref{eq:semicircle_pde}--\eqref{eq:semicircle_pde1}, where
$\mathcal{Q}^{\mathcal{M}}= \mathcal{Q}^{\mathcal{M}}(r,s)$ is a
polynomial function of the independent variables $r$ and $s$.

The proof now proceeds inductively. The induction basis is given by
the asymptotics~\eqref{eq:u_xi_mu} of the function $U$. To complete
the proof we thus need to establish the following inductive step: {\em
  provided that} for some integer $\mathcal{L}$ the function $U$ can
be expressed in the form
\begin{equation}
\label{eq:induction_step}
  U(\xi) = \mathcal{P}^{\mathcal{L}}(\xi^{1/2},\overline{\xi}^{1/2}) + o(\xi^{\mathcal{\mathcal{L}} - \lambda})\quad\mbox{as $\xi\to 0$},
\end{equation}
for some $0<\lambda<1$, where $ \mathcal{P}^{\mathcal{L}}=
\mathcal{P}^{\mathcal{L}}(r,s)$ is a polynomial function of the
independent variables $r$ and $s$, {\em then} a similar relation holds
for $U$ with an error of order $o(\xi^{\mathcal{L}+1 -
  \lambda})$ and for a certain polynomial
$\mathcal{P}^{\mathcal{L}+1}(\xi^{1/2},\overline{\xi}^{1/2})$:
\begin{equation}\label{eq:u_l_plus_1}
  U(\xi) = \mathcal{P}^{\mathcal{L}+1}(\xi^{1/2},\overline{\xi}^{1/2}) + o(\xi^{\mathcal{\mathcal{L}}+1- \lambda}).
\end{equation}
To do this we apply Lemma~\eqref{Lemma 3} to each term on the right
hand side of equation~\eqref{final_representation}. For the terms
including the operator $\Lambda_1$, for example, such an asymptotic
representation with error of the order
$o(\xi^{\mathcal{\mathcal{L}}+1- \lambda})$ can be obtained by using
the assumption~\eqref{eq:induction_step} and the $\mathcal{L}$-th
order Taylor expansion of the smooth function $K(t)$ around $t=0$, and
by applying equations~\eqref{eq:lambda_1} with $\beta = 0, 1/2,
1,\dots,\mathcal{L}$, $\gamma = 0, 1/2, 1,\dots,\mathcal{L}$ and
equation~\eqref{eq:lambda_1_rem} with $\gamma = \mathcal{L} - \lambda$
to the resulting polynomial and error terms for the product
$K(t)U(t)$. The terms that contain the operators $\Lambda_2$ and
$\Lambda_3$ can be treated similarly on the basis of Taylor expansions
of the functions $F(t)$ and $G(t)$ around $t=0$ and application of
equations~\eqref{eq:lambda_2},~\eqref{eq:lambda_3} with $\beta =
1,\dots,\mathcal{L}$, $\alpha = 1,\dots,\mathcal{L}$ , and
equations~\eqref{eq:lambda_2_rem},~\eqref{eq:lambda_3_rem} with
$\gamma = \mathcal{L} - \lambda$ and $\alpha = \mathcal{L} -
\lambda+1$. The inductive step and therefore the proof of
Theorem~\ref{th:solution_singularity} are thus complete.
\end{proof}

\subsection{Proof of theorem~\ref{th:density_singularity}}
\label{sec:density_singularities}


The relationship between the density $\psi$ and the PDE solution $u$
is given by equation~\eqref{eq:density_representation} if $\Omega$ is
an exterior domain and by
equation~\eqref{eq:density_representation_interior} if $\Omega$ is an
interior domain.  Throughout this section we assume that $\Omega$ is
an interior domain, and, thus, that $\psi$ is given by
equation~\eqref{eq:density_representation_interior}; the proof for
exterior domains is analogous.

In order to establish the singular character of the density $\psi$ we
first seek an asymptotic expression for the conjugate solution $u_e$
near $z=0$ (see Definition~\ref{def:exterior_conjugate}). Using a conformal mapping approach for $u_e$ similar to
the one used in Section~\ref{sec:conformal_mapping} for the solution $u$ of the problem~\eqref{eq:exterior}, in this case we employ a conformal map $z = v(\xi)$ which maps the semi-circular region
$D_A$ depicted in Figure~\ref{fig:semicircle} in the complex $\xi$-plane onto
the domain $\overline{B(y^0,r) \setminus \Omega^0}$ in the complex
$z-$plane (see Figure~\ref{fig:singularity}). We assume, as we may,
that $v$ maps the origin to itself and that the intervals
$\{\textrm{Im}(\xi)=0, 0 \le \textrm{Re}(\xi) \le A\}$ and
$\{\textrm{Im}(\xi)=0, -A \le \textrm{Re}(\xi) \le 0\}$ are mapped
onto the boundary segments $ \Gamma^0_{q_1}$ and $ \Gamma^0_{q_2}$,
respectively (see equation~\eqref{eq:omega0}). 
Following Section~\ref{sec:conformal_mapping}  in this case we introduce the function $V(\xi) = u_e(v(\xi))$ and we note that $V$ satisfies the second order elliptic problem (cf. \cite[eq. 5.4.17]{fokas2003complex})
\begin{align}
\label{eq:semicircle_pde_conjugate}
&\Delta  V+ K_1(\xi) V =0 &\quad&\mbox{for}\quad \xi\in \textrm{int}(D_A),\\
&V(\xi)=u_e(v(\xi)) &\quad&\mbox{for}\quad \textrm{Im}(\xi)=0,\quad\mbox{and}\label{eq:semicircle_pde_conjugate0} \\
&V(\xi)=M_1(\xi) &\quad&\mbox{for}\quad |\xi| =A,\label{eq:semicircle_pde_conjugate1}
\end{align}
where $M_1$ is given by $M_1(\xi)=u_e(v(\xi))$. 

The following Lemma parallels Lemma~\ref{lemma:u_xi_mu}.
\begin{lemma}\label{lemma:v_xi_mu}
The solution $V$ of
the problem~\eqref{eq:semicircle_pde_conjugate}--\eqref{eq:semicircle_pde_conjugate1} satisfies the asymptotic relation
\begin{equation}
\label{eq:V_zeta_mu}
V(\xi)=o\left(\xi^{\mu }\right)
\end{equation}
for all $-\frac{1}{2}<\mu <0$.
\end{lemma}
\begin{proof}
Employing the Laplace Green function 
\begin{equation}
\label{greensfunction_dirichlet}
G_1(t,\xi)=-\frac{1}{2 \pi }\left\{  \log \left|t - \xi\right| - \log \left|t - \overline{\xi}\ \right| \right\}
\end{equation}
for the Dirichlet problem~\eqref{eq:semicircle_pde_conjugate}--\eqref{eq:semicircle_pde_conjugate1} 
and applying the Green formula to the functions $V$ and $G_1$ on the domain $D_A$ we obtain
\begin{equation}
\label{eq:green_formula_V_DA}
V(\xi)=\int \int _{D_A}G_1(t,\xi) \Delta  V(t) d x_t d y_t+\int _{\partial D_A}\left\{ V(t) \frac{\partial }{\partial  n_t}G_1(t,\xi)-G_1(t,\xi)
\frac{\partial }{\partial  n_t}V(t)\right\}ds_t.
\end{equation}
Since $G_1(t,\xi)=0$ for $\textrm{Im}(\xi) = 0$ and since $\partial D_A = [-A,A] \cup \Gamma_A$, the triangle inequality yields
\begin{equation}
\begin{split}
\label{V_bound}
|V(\xi)| \le &\left| \int \int _{D_A} G_1(t,\xi) K_1(t) V(t) d x_t d y_t \right|  + \left| \int^A_{-A} V(t) \frac{\partial }{\partial  n_t}G_1(t,\xi)dt \right|\\
 +& \left|\int _{\Gamma_A}\left\{ V(t) \frac{\partial }{\partial  n_t}G_1(t,\xi)-G_1(t,\xi)
\frac{\partial }{\partial  n_t}V(t)\right\}ds_t \right|.
\end{split}
\end{equation}
For the integral over the outer arc $\Gamma_A$ in~\eqref{V_bound} we have
 \begin{equation}
   \left|\int _{\Gamma_A}\left\{ V(t) \frac{\partial }{\partial  n_t}G_1(t,\xi)-G_1(t,\xi)
       \frac{\partial }{\partial  n_t}V(t)\right\}ds_t \right| \le C_1\quad\mbox{for}\quad |\xi| < A/2,
\end{equation}
where $C_1$ is a nonnegative constant, as it can be checked easily in view of
the boundedness of the integrands for $\xi$ near the origin.  From
Lemma~\ref{def_rem}, further, it easily follows that $V\in H^1(D_A)$,
and thus, bounding the absolute value of the first integral in
equation~\eqref{V_bound} by means of the Cauchy-Schwarz inequality we
obtain the bound
\begin{equation}
\left| \int \int _{D_A} G_1(t,\xi) K_1(t) V(t) d x_t d y_t \right| \le ||G_1||_{L^2(D_A)} ||V||_{L^2(D_A)} \max_{t \in D_A}(K_1(t))
\end{equation}
for all $\xi\in D_A$.  As is well known, finally, double layer
potentials for bounded densities are uniformly bounded in all of space
(see e.g.~\cite[Lemma 3.20]{folland1995introduction}). It follows that
the second integral in equation~\eqref{V_bound} is uniformly bounded
for $\xi\in\mathbb{R}^2$ since, in view
of~\eqref{eq:semicircle_pde_conjugate0},
Definition~\ref{def:exterior_conjugate} and
Theorem~\ref{th:solution_singularity}, $V$ is a bounded function for
$t_2 = \textrm{Im}(t) = 0$. The relation~\eqref{eq:V_zeta_mu} thus
follows for all $-\frac{1}{2}<\mu <0$ and the proof is complete.
\end{proof}

\begin{corollary}\label{cor2}
The derivatives of the solution $V$ of
the problem~\eqref{eq:semicircle_pde_conjugate}--\eqref{eq:semicircle_pde_conjugate1} satisfy the asymptotic relation
\begin{equation}
\label{eq:dv_xi_mu}
D^h V = o\left(\xi^{\mu -h}\right)
\end{equation}
for all $-\frac{1}{2}<\mu <0$.
\end{corollary}

\begin{proof}
See~\cite[Section 4]{wigley1964}
\end{proof}

We now proceed with the main proof of this section, which is based on
an inductive argument similar to the one used in the proof of
Theorem~\ref{th:solution_singularity}.
\begin{proof} [Proof of Theorem~\ref{th:density_singularity}]
  Applying the Green formula on the set $D_{A,\delta}$ (equation~\eqref{D_a_delta})
and letting $\delta \to 0$ we obtain
\begin{equation}
\label{eq:final_form_V}
V(\xi) = \int\int_{D_A}\left( - K_1(t) V(t) \right)G_1(t,\xi) dt -\int_{-A}^A V(t)\frac{\partial}{\partial n_t}G_1(t,\xi) dt+p_1(\xi) + p_2(\overline{\xi}),
\end{equation}
where $p_1$ and $p_2$ denote power series with positive radii of
convergence. In view of equation~\eqref{eq:semicircle_pde_conjugate0},
Definition~\ref{def:exterior_conjugate} and
Theorem~\ref{th:solution_singularity}, on the other hand, we see that, for any given
integer $\mathcal{L}$, the boundary values of $V$ at $\xi_2=0$ satisfy
\begin{equation}\label{eq:v_on_the_boundary}
  V(\xi_1,0) = V(\xi) = u_e(\xi) = \mathcal{P}_{y_0}^\mathcal{L}(\xi^{1/2},\overline{\xi}^{1/2}) + o(\xi^{\mathcal{L}})\quad \mbox{for}\quad \xi_2 =\textrm{Im}(\xi)=0
\end{equation}
(see ~\eqref{final_form}). Relying on
equations~\eqref{eq:final_form_V} and~\eqref{eq:v_on_the_boundary} as
well as Lemma~\ref{Lemma 3}, an inductive argument similar to the one
used in the proof of Theorem~\ref{th:solution_singularity} shows that
for any integer $\mathcal{N}$ the function $V$ satisfies an asymptotic
relation of the form
\begin{equation}\label{V_asymptotics}
V(\xi) =\mathcal{P}^{\mathcal{N}}(\xi^{1/2},\overline{\xi}^{1/2}) + o(\xi^{\mathcal{N}})\quad\mbox{as $\xi\to 0$},
\end{equation}
where $\mathcal{P}^{\mathcal{N}}$ is an $\mathcal{N}-$dependent
polynomial. In view of corollaries~\ref{cor1} and~\ref{cor2},
substitution of the normal derivatives of
equations~\eqref{V_asymptotics} and~\eqref{final_form} for
$\textrm{Im}(\xi)=0$ into
equation~\eqref{eq:density_representation_interior} yields
\begin{equation}\label{eq:psi_no_sides}
  \psi(\xi) = \xi^{-1/2}\mathcal{Q}^{\mathcal{N}}_1(\xi^{1/2},\overline{\xi}^{1/2}) + \overline{\xi}^{-1/2}\mathcal{Q}^{\mathcal{N}}_2(\xi^{1/2},\overline{\xi}^{1/2}) + o(\xi^{\mathcal{\mathcal{N}} -1})\quad \mbox{for}\quad \xi_2 =\textrm{Im}(\xi)=0,
\end{equation}
where ${Q}^{\mathcal{N}}_i$ are $\mathcal{N}-$dependent polynomials.
The desired asymptotic relations~\eqref{eq:final_form_density} now
follow by re-expressing~\eqref{eq:psi_no_sides} in terms of the
distance function $d$, and the proof is thus complete.
\end{proof}

\section{Singularity resolution via Fourier Continuation}
\label{sec:smooth_geometries}
Theorem~\ref{th:solution_singularity} tells us that the solutions of
Zaremba problem~\eqref{eq:exterior} possess a very specific
singularity structure near the Dirichlet-Neumann junctions---which, as
shown in Theorem~\ref{th:density_singularity}, are inherited by the
solutions of the corresponding integral equation
system~\eqref{eq:exterior_integral_system}. In particular,
equation~\eqref{eq:final_form_density} shows that the integral
equation solutions can be expressed as a product of the function
$1/\sqrt{d}$ and a smooth function of $\sqrt{d}$, where $d$ denotes
the distance to the Dirichlet-Neumann junction.

The question thus arises as to how to incorporate the singular
characteristics of the integral equation solutions in order to design
a numerical integration method of high order of accuracy for the
numerical discretization of the integral equation
system~\eqref{eq:exterior_integral_system}. A relevant reference in
these regards is provided by the contribution~\cite{bruno2012second}
(see also~\cite{yan1988integral}), which provides a high-order solver
for the problem of scattering by open arcs.  As is known, the open-arc
integral equation solutions possess singularities around the
end-points: they can be expressed as a product of the function
$1/\sqrt{d}$ and a smooth function of $d$---or, in other words, the
asymptotics of the integral solutions are functions which {\em only
  contain powers of $\sqrt{d}$ with exponents equal to } $(2n-1)$ for
$n \ge 0$. In particular~\cite{bruno2012second} a change of variables
of the form $t= \cos{s}$, $0 \le s \le \pi$ in parameter space
completely regularizes the problem, and it thus gives rise to
spectrally accurate numerical approximations of the form
\begin{equation}
\psi \sim \sum_{j=0}^n C_j \cos(j s) \qquad 0\le s\le\pi
\end{equation}
for the integral-equation solutions $\psi$.  

As shown in Theorem~\eqref{th:density_singularity}, on the other hand,
the asymptotic expansions of the integral-equation solutions $\psi$
considered in this paper contain {\em all integer powers of }
$\sqrt{d}$, and therefore, as established in~\cite{AkhBrunoNigam}, a
cosine change of variables such as the one considered above leads to a
full Fourier series---containing all $2\pi$-periodic cosines and
sines,
\begin{equation}\label{eq:FC_approximation}
\psi \sim \sum_{j=0}^n C_j \cos(j s) +  D_j \sin(j s) \qquad 0\le s\le\pi,
\end{equation}
even though values for $\psi$ can only be determined for $0\leq s \leq
\pi$.  The key element that allows such expansions in the extended
interval $[0,2\pi]$ is the Fourier Continuation (FC) method introduced
in~\cite{bruno2009_1,albin2011} and first suitably generalized to the
present context in~\cite{AkhBrunoNigam}. This leads to a Fourier
series that converges with high-order accuracy to the integral density
$\psi $ in the interval $[0,\pi]$.

Exploiting such rapidly convergent Fourier expansions, a numerical
method for  Zaremba  boundary value problems is presented in the following section.  

\section{Numerical algorithms}
\label{num_alg}

Closed form expressions are presented in~\cite{AkhBrunoReitich} for
the integrals of products trigonometric functions and logarithms that
appear in equation~\eqref{eq:exterior_integral_system} upon
substitution of the expansion~\eqref{eq:FC_approximation}; as shown
in~\cite{AkhBrunoNigam}, use of such expressions leads to highly
accurate approximations of the integral operators in
equation~\eqref{eq:exterior_integral_system}.  In particular, these
algorithms rely on Nystr\"om discretization of the solution $\psi$; in
view of the structure of the integrand, the discretization points used
are given as the image of a uniform mesh in $s$ variable under the
change of variables $t = \cos(s)$. The Fourier
series~\eqref{eq:FC_approximation} is  obtained via application of
the Fourier Continuation method in the $s$ variable, and a discrete
version of the integral equation system is thus obtained on the basis
of a resulting matrix $\mathcal{A}$.

These discrete operators were used in reference~\cite{AkhBrunoNigam}
as important building blocks of an efficient algorithm for evaluation
of eigenvalues and eigenfunctions of the Laplace operator under
Zaremba boundary conditions. Figure~\ref{fig:hf_eigenfunction}
presents high-frequency Zaremba eigenfunctions obtained by means of
that solver. In what follows we use these discrete operators as well
as the vector $\tt{f}$ of values of given functions $f$ and $g$ at the
aforementioned discretization points to produce solutions of the
Zaremba problem~\eqref{eq:exterior}. For simplicity we rely on an
$LU$-decomposition applied to the linear system $\mathcal{A} \tt{c} =
\tt{f}$ to obtain a discrete approximation $\tt{c}$ of the solution
$\psi$. Note, however, that as mentioned in Theorem~\ref{th:exterior},
the integral equation system~\eqref{eq:exterior_integral_system} is
not invertible for a certain discrete set of values of $k$ (spurious
resonances) that correspond to Dirichlet eigenvalues on the
complementary set $\mathbb{R}^2 \setminus \Omega$. A numerical
methodology described in Appendix~\ref{sec:uniqueness} extends
applicability of the proposed solvers to all frequencies, including
spurious resonances.  A variety of numerical results
obtained by means of the aforementioned Zaremba boundary-value
solvers are presented in the following section.

\begin{figure}[h!]
\centering
  \includegraphics[width=6in]{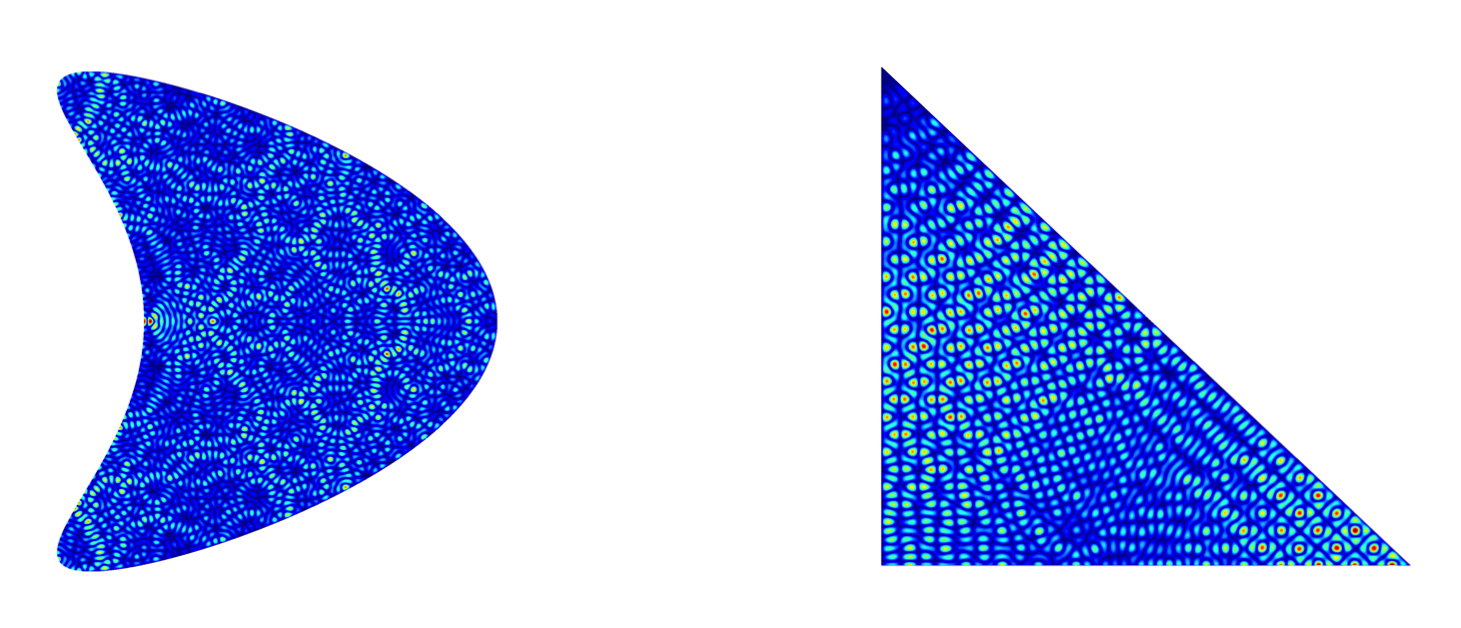}
  \caption{High frequency Zaremba eigenfunctions.}
  \label{fig:hf_eigenfunction}
\end{figure}

\section{Numerical results\label{num_res}}

This section presents results of applications of the new solvers to
problems of scattering by smooth obstacles under boundary conditions
of Zaremba type. This entails solution of the
problem~\eqref{eq:exterior} for exterior domains $\Omega$ and for
which the right hand sides $f$ and $g$ are given by
\begin{equation}
\begin{split}
\label{scattering_problem}
f  & =  e^{i k \alpha \cdot x} = e^{i k (\cos(\alpha) x_1 +\sin(\alpha) x_2 ) )}\\
g &= n_x \cdot \nabla e^{i k \alpha \cdot x},
\end{split}
\end{equation}
where $\alpha$ is the angle of incidence. 

In our first experiment we apply the solver to a kite-shaped domain
whose boundary is given by the parametrization
\begin{eqnarray}\label{eq:kite}
 x_1=\cos(t)+0.65 \cos(2t)-0.65\quad\mbox{ and }\quad x_2=1.5 \sin(t), 
\end{eqnarray} 
assuming the Neumann and Dirichlet boundary portions $\Gamma_N$ and
$\Gamma_D$ corresponding to $ t \in [\pi/2; 3\pi/2]$ and its
complement, respectively.  Figure~\ref{fig:kite_scattering} depicts
the scattered and total fields that result as a wave with wavenumber
$k=40$ and incidence angle $\alpha=\pi/8$ impinges on the
object. Figure~\ref{fig:convergence} demonstrates the high-order
convergence results for the value of the scattered field $u(x_0)$ at
the point $x_0 = (1,2)$ in the exterior of the domain.

\begin{figure}[H]
\centering
\includegraphics[width=7in]{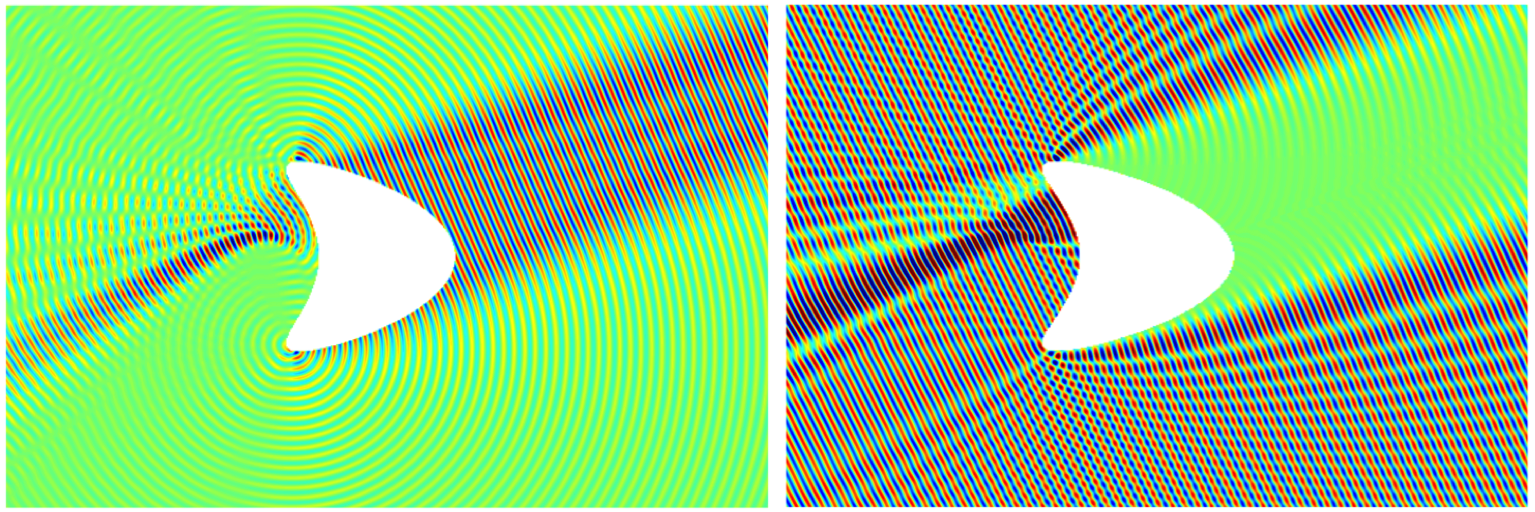}
  \caption{Scattering from a kite-shaped domain under Zaremba boundary
  conditions. Left: Scattered
    field. Right: Total field.}
  \label{fig:kite_scattering}
\end{figure}

\begin{figure}[h]
\centering
 \includegraphics[width=5in]{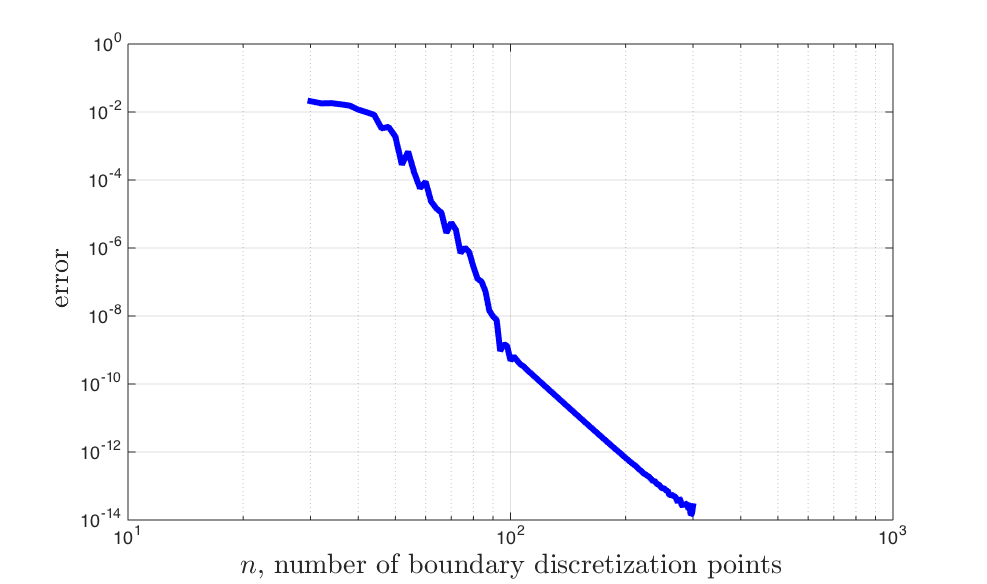}
  \caption{Convergence of the value $u(x_0)$ ($x_0 =(1,2)$) for a kite
    shaped domain with $k=10$.}
  \label{fig:convergence}
\end{figure}

A similar example concerns scattering by the unit disc under Dirichlet
and Neumann boundary conditions prescribed on the left and right
halves of the disc boundary, respectively.
Figure~\ref{fig:disc_scattering} displays  the scattered and
total fields that result as a wave with wavenumber $k=50$ and
incidence angle $\alpha=\pi/8$ impinges on the disc.
\begin{figure}[H]
\centering
\includegraphics[width=6in]{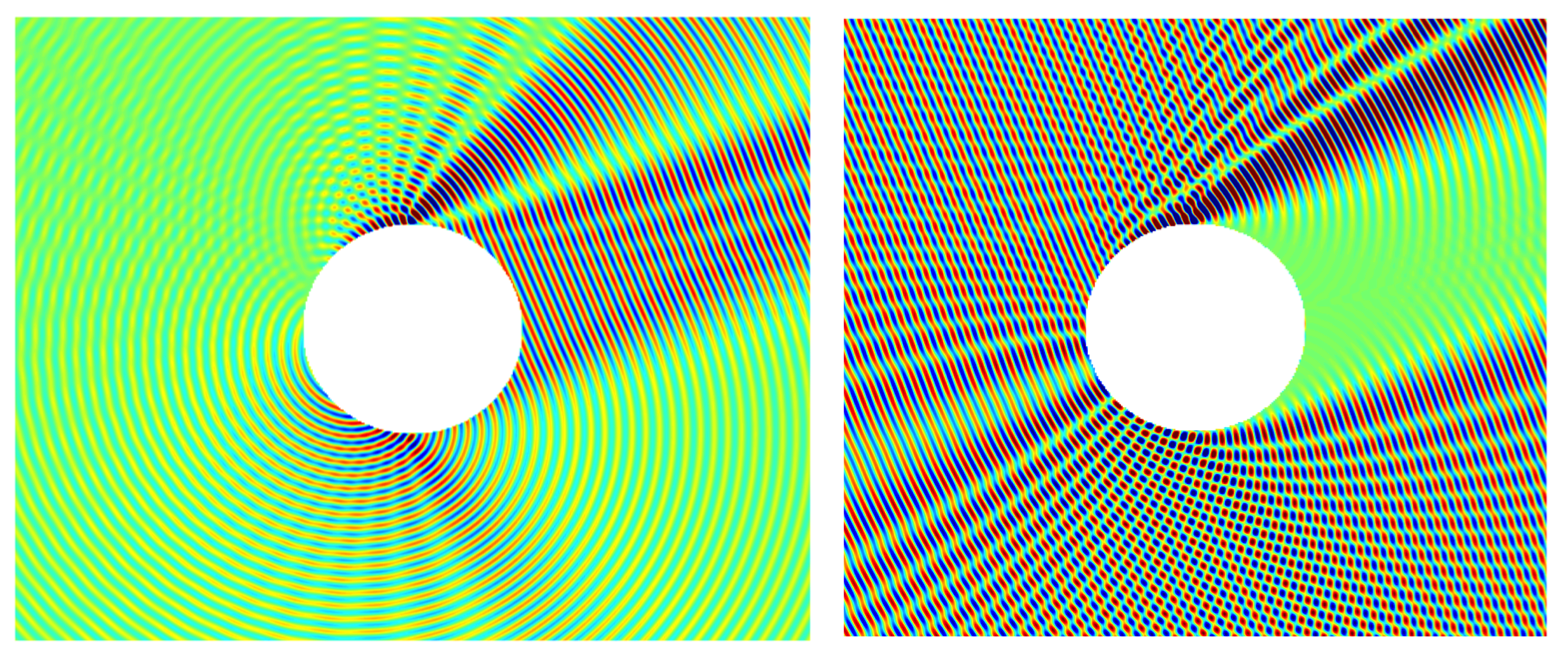}
\caption{Scattering from a disc under Zaremba boundary
  conditions. Left: Scattered
    field. Right: Total field.}
  \label{fig:disc_scattering}
\end{figure}

To conclude this section we present a brief comparison of the proposed
solvers with one of the most efficient Zaremba solvers previously
available~\cite{helsing2009,helsing2013solving}. This method, which is
based on iterative inverse preconditioning and which applies to a
variety of singular problems, has been implemented in a numerical
MATLAB package which is freely available~\cite{helsingurl}.
Unfortunately, Zaremba boundary conditions are not implemented in that
package. At any rate, numerical experiments suggest that the execution
times required by the algorithm~\cite{helsingurl} are not as favorable
as those required by the solvers proposed in this paper. Indeed, the
solver~\cite{helsingurl} (which is applicable to domains with corners)
required a computing time of 0.46 seconds to solve the Dirichlet
problem for Helmholtz equation with wavenumber $k=2$ on the unit disc
(whose boundary was partitioned into two Dirichlet arcs) by mean of
GMRES iterations with a GMRES residual of $10^{-13}$, while a
GMRES-based implementation of the FC-based solver presented in this
paper runs in 0.06 seconds for the significantly more challenging
Zaremba problem for the Helmholtz equation on the same domain, with
the same incident wave frequency and with the same GMRES
residual. (These and all other numerical results presented in this
section were obtained on a single core of a 2.4 GHz Intel E5-2665
processor.)  Such time differences, a factor of eight in this case,
can be very significant in practice, in contexts where thousands or
even tens of thousands of solutions are necessary, as is the case in
inverse problems as well as in our own solution~\cite{AkhBrunoNigam}
of high-frequency eigenvalue problems, etc. The main reason for the
difference in execution times is that even though the iterative
solver~\cite{helsingurl} requires a limited number of iterations,
certain iteration-dependent matrix entries occur in this solver (in
view of corresponding iteration dependent discretization points it
uses), which require large number of evaluations of expensive Hankel
functions at each iteration, and, thus, a significant computing cost
per iteration.

\section{Conclusions\label{concl}}

This paper presented, for the first time, detailed asymptotic
expansions near Dirichlet-Neumann junctions for solutions of Zaremba
problems on smooth two dimensional domains.  By precisely accounting
for the singularities of the boundary densities and kernels on the
basis of Fourier-Continuation expansions, further, discretizations of
high-order accuracy were obtained for relevant boundary integral
operators.  The resulting integral-equation solver allows for accurate
and efficient approximation of the highly-singular Zaremba solutions 
at both the low- and high-frequency regimes.

\section*{Acknowledgments}
The authors gratefully acknowledge support from AFOSR and NSF. 
\appendix

\section{Appendix: Exterior solution at  interior resonances} 
\label{sec:uniqueness}

This appendix describes an algorithm for evaluation of the solution of
the problem~\eqref{eq:exterior} for an exterior domain $\Omega$, and
for a value of $k^2$ that either equals or is close to an interior
Dirichlet eigenvalue of the Laplace operator in the bounded set $\R^2
\setminus \Omega$. As mentioned in
Section~\ref{sec:integral_equations} in this case the system of
integral equations~\eqref{eq:exterior_integral_system} does not a have
a unique solution. However, the solution of the PDE is uniquely
solvable for any value of $k$.

The non-invertibility of the aforementioned continuous systems of
integral equations at a wavenumber $k = k^*$ manifests itself at the
discrete level in non-invertibility or ill-conditioning of the system
matrix $\mathcal{A}:=\mathcal{A}(k)$ for values of $k$ close to
$k^*$. Therefore, for $k$ near $k^*$ the numerical solution of the
Zaremba problems under consideration (which in what follows will be
denoted by $\tilde u:=\tilde u_k(x)$ to make explicit the solution
dependence on the parameter $k$) cannot be obtained via direct
solution the linear system $\mathcal{A}(k) \cdot \tt{c} =\tt{f}$. As is known,
however, the solutions $u = u_k$ of the continuous boundary value
problem are analytic functions of $k$ for all real values of
$k$---including, in particular, for $k$ equal to any one of the
spurious resonances mentioned above---and therefore, the approximate
values $\tilde u_k(x)$ for $k$ sufficiently far from $k^*$ can be
used, via analytic continuation, to obtain corresponding
approximations around $k = k^*$ and even at a spurious resonance~$k =
k^*$.

In order to implement this strategy for a given value of $k = k_0$ it is
necessary for our algorithm to possess a capability to perform two
steps:
\begin{enumerate}
\item To determine whether $k_0$ is ``sufficiently far'' from any one of
  the spurious resonances~$k^*$.
\item[2a.] If $k_0$ is ``sufficiently far'' then simply invert the linear
  system by means of either an $LU$ decomposition or the usually
  already available Singular Value Decomposition (which is used to
  determine the ``distance from resonance'').
\item[2b.] If $k_0$ is not ``sufficiently far'' from one of the spurious
  resonances~$k^*$, then obtain the PDE solution at $k_0$ by
  analytic continuation from solutions for values of $k$ in a
  neighborhood of $k_0$ which are ``sufficiently far'' from~$k^*$.
\end{enumerate}
Here $k$ is said to be ``sufficiently far'' from the set of spurious
resonances provided the corresponding linear system can be inverted
without significant error amplifications. It has been noticed in
practice~\cite{perez2014high} that the regions within which inversion
is not possible are very small indeed, in such a way that analytic
continuation from ``sufficiently far'' can be performed to the
singular or nearly singular frequency $k_0$ with any desired
accuracy. For full details in these regards see~\cite{perez2014high}.

Numerical results confirming highly accurate evaluation of the PDE
solution even for resonant frequencies are presented in
Figure~\ref{fig:resonant_freq_convergence} for the case of the
FC-based solver applied to the Zaremba boundary value problem on the
unit disc.  The solution errors are displayed for two frequencies:
$k$ = 11 (where the solutions are obtained using an LU decomposition)
and the resonant frequency $k=11.791534439014281$ (with solutions
obtained by means of analytic continuation). Clearly, the proposed
approach can tackle the spurious-resonance problem without difficulty.

\begin{figure}[h]
  \centering
  \includegraphics[width=6in]{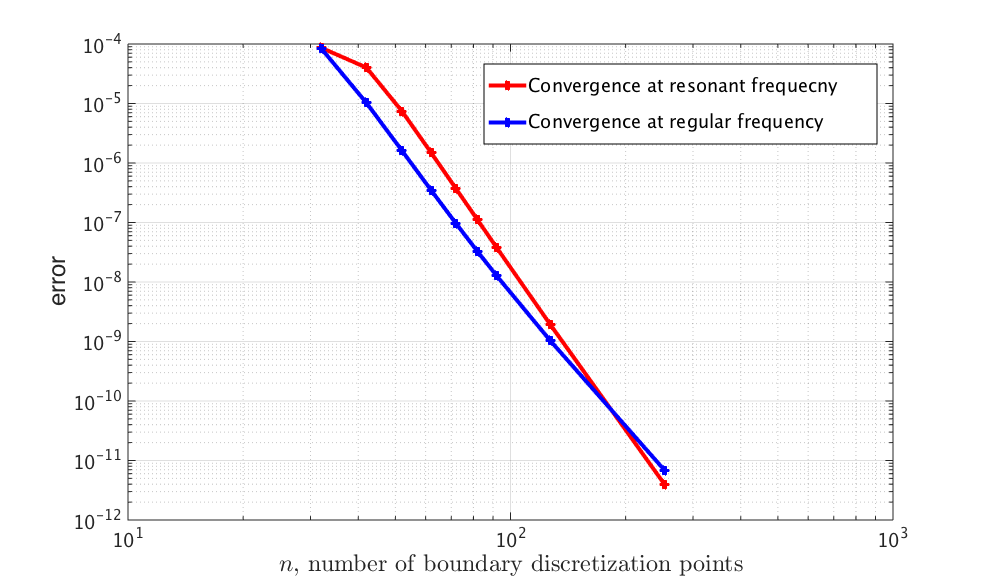}
  \caption{Convergence comparison at a regular and a resonant frequency.}
  \label{fig:resonant_freq_convergence}
\end{figure}

\bibliographystyle{acm}
\bibliography{Library}

\begin{thebibliography}{10}

\bibitem{AkhBrunoNigam}
{\sc Akhmetgaliyev, E., Bruno, O.~P., and Nigam, N.}
\newblock A boundary integral algorithm for the laplace dirichlet–neumann
  mixed eigenvalue problem.
\newblock {\em Journal of Computational Physics 298\/} (2015), 1--28.

\bibitem{AkhBrunoReitich}
{\sc Akhmetgaliyev, E., Bruno, O.~P., and Reitich, F.}
\newblock Integral equation solution of mixed boundary-value problems: domain
  smoothing and singularity resolution.
\newblock In preparation.

\bibitem{albin2011}
{\sc Albin, N., and Bruno, O.~P.}
\newblock A spectral {FC} solver for the compressible {Navier-Stokes} equations
  in general domains {I}: {Explicit} time-stepping.
\newblock {\em Journal of Computational Physics 230\/} (2011), 6248--6270.

\bibitem{britt2013high}
{\sc Britt, D., Tsynkov, S., and Turkel, E.}
\newblock A high-order numerical method for the helmholtz equation with
  nonstandard boundary conditions.
\newblock {\em SIAM Journal on Scientific Computing 35}, 5 (2013),
  A2255--A2292.

\bibitem{bruno2012second}
{\sc Bruno, O.~P., and Lintner, S.~K.}
\newblock Second-kind integral solvers for {TE} and {TM} problems of
  diffraction by open arcs.
\newblock {\em Radio Science 47}, 6 (2012).

\bibitem{bruno2009_1}
{\sc Bruno, O.~P., and Lyon, M.}
\newblock High-order unconditionally stable {FC-AD} solvers for general smooth
  domains {I. Basic elements}.
\newblock {\em Journal of Computational Physics 229\/} (2009), 2009--2033.

\bibitem{cakoni2005boundary}
{\sc Cakoni, F., Hsiao, G.~C., and Wendland, W.~L.}
\newblock On the boundary integral equation method for a mixed boundary value
  problem of the biharmonic equation.
\newblock {\em Complex Variables, Theory and Application: An International
  Journal 50}, 7-11 (2005), 681--696.

\bibitem{chang2014edge}
{\sc Chang, D.-C., Habal, N., and Schulze, B.-W.}
\newblock The edge algebra structure of the zaremba problem.
\newblock {\em Journal of Pseudo-Differential Operators and Applications 5}, 1
  (2014), 69--155.

\bibitem{COLTON:1998}
{\sc Colton, D.~L., and Kress, R.}
\newblock {\em Inverse Acoustic and Electromagnetic Scattering Theory}.
\newblock Springer, 1998.

\bibitem{duduchava2013mixed}
{\sc Duduchava, R., and Tsaava, M.}
\newblock Mixed boundary value problems for the helmholtz equation in arbitrary
  2d-sectors.
\newblock {\em Georgian Mathematical Journal 20}, 3 (2013), 439--467.

\bibitem{duduchava2015mixed}
{\sc Duduchava, R., and Tsaava, M.}
\newblock Mixed boundary value problems for the laplace-beltrami equations.
\newblock {\em arXiv preprint arXiv:1503.04578\/} (2015).

\bibitem{fabrikant1991mixed}
{\sc Fabrikant, V.}
\newblock {\em Mixed boundary value problems of potential theory and their
  applications in engineering}, vol.~68.
\newblock Kluwer Academic Pub, 1991.

\bibitem{fichera1949analisi}
{\sc Fichera, G.}
\newblock Analisi esistenziale per le soluzioni dei problemi al contorno misti,
  relativi all'equazione e ai sistemi di equazioni del secondo ordine di tipo
  ellittico, autoaggiunti.
\newblock {\em Annali della Scuola Normale Superiore di Pisa-Classe di Scienze
  1}, 1-4 (1949), 75--100.

\bibitem{fichera1952sul}
{\sc Fichera, G.}
\newblock Sul problema della derivata obliqua e sul problema misto per
  l'equazione di laplace.
\newblock {\em Bollettino dell'Unione Matematica Italiana 7}, 4 (1952),
  367--377.

\bibitem{fokas2003complex}
{\sc Fokas, A.~S.}
\newblock {\em Complex variables: introduction and applications}.
\newblock Cambridge University Press, 2003.

\bibitem{folland1995introduction}
{\sc Folland, G.~B.}
\newblock {\em Introduction to partial differential equations. 2nd edition}.
\newblock Princeton University Press, 1995.

\bibitem{giraud1934problemes}
{\sc Giraud, G.}
\newblock {\em Probl{\`e}mes mixtes et Probl{\`e}mes sur des vari{\'e}t{\'e}s
  closes, relativement aux {\'e}quations lin{\'e}aires du type elliptique}.
\newblock Imprimerie de l'Universit{\'e}, 1934.

\bibitem{helsing2009}
{\sc Helsing, J.}
\newblock Integral equation methods for elliptic problems with boundary
  conditions of mixed type.
\newblock {\em Journal of Computational Physics 228:23\/} (2009), 8892--8907.

\bibitem{helsing2013solving}
{\sc Helsing, J.}
\newblock Solving integral equations on piecewise smooth boundaries using the
  rcip method: a tutorial.
\newblock In {\em Abstract and Applied Analysis\/} (2013), vol.~2013, Hindawi
  Publishing Corporation.

\bibitem{helsingurl}
{\sc Helsing, J.}
\newblock {Solving integral equations on piecewise smooth boundaries using the
  RCIP method: a tutorial.}
\newblock \url{http:http://www.maths.lth.se/na/staff/helsing/Tutor/index.html},
  2013.

\bibitem{lorenzi1975mixed}
{\sc Lorenzi, A.}
\newblock A mixed boundary value problem for the laplace equation in an angle
  (1st part).
\newblock {\em Rendiconti del Seminario Matematico della Universit{\`a} di
  Padova 54\/} (1975), 147--183.

\bibitem{magenes1955sui}
{\sc Magenes, E.}
\newblock Sui problemi di derivata obliqua regolare per le equazioni lineari
  del secondo ordine di tipo ellittico.
\newblock {\em Annali di Matematica Pura ed Applicata 40}, 1 (1955), 143--160.

\bibitem{mclean2000strongly}
{\sc McLean, W. C.~H.}
\newblock {\em Strongly elliptic systems and boundary integral equations}.
\newblock Cambridge university press, 2000.

\bibitem{perez2014high}
{\sc P{\'e}rez-Arancibia, C., and Bruno, O.~P.}
\newblock High-order integral equation methods for problems of scattering by
  bumps and cavities on half-planes.
\newblock {\em JOSA A 31}, 8 (2014), 1738--1746.

\bibitem{signorini1916sopra}
{\sc Signorini, A.}
\newblock Sopra un problema al contorno nella teoria delle funzioni di
  variabile complessa.
\newblock {\em Annali di Matematica Pura ed Applicata (1898-1922) 25}, 1
  (1916), 253--273.

\bibitem{warschawski1961differentiability}
{\sc Warschawski, S.}
\newblock On differentiability at the boundary in conformal mapping.
\newblock {\em Proceedings of the American Mathematical Society 12}, 4 (1961),
  614--620.

\bibitem{warschawski1935higher}
{\sc Warschawski, S.~E.}
\newblock On the higher derivatives at the boundary in conformal mapping.
\newblock {\em Transactions of the American Mathematical Society 38}, 2 (1935),
  310--340.

\bibitem{wendland1979integral}
{\sc Wendland, W.~L., Stephan, E., Hsiao, G.~C., and Meister, E.}
\newblock On the integral equation method for the plane mixed boundary value
  problem of the laplacian.
\newblock {\em Mathematical Methods in the Applied Sciences 1}, 3 (1979),
  265--321.

\bibitem{wigley1964}
{\sc Wigley, N.~M.}
\newblock Asymptotic expansions at a corner of solutions of mixed boundary
  value problems.
\newblock {\em Journal of Mathematics and Mechanics 13\/} (1964), 549--576.

\bibitem{wigley1970mixed}
{\sc Wigley, N.~M.}
\newblock Mixed boundary value problems in plane domains with corners.
\newblock {\em Mathematische Zeitschrift 115}, 1 (1970), 33--52.

\bibitem{yan1988integral}
{\sc Yan, Y., Sloan, I.~H., et~al.}
\newblock {\em On integral equations of the first kind with logarithmic
  kernels}.
\newblock University of NSW, 1988.

\bibitem{zaremba1910probleme}
{\sc Zaremba, S.}
\newblock Sur un probleme mixte relatifa l’{\'e}quation de laplace.
\newblock {\em Bulletin de l’Acad{\'e}mie des sciences de Cracovie, Classe
  des sciences math{\'e}matiques et naturelles, s{\'e}rie A\/} (1910),
  313--344.

\end{thebibliography}

\end{document}